\documentclass[11pt, reqno]{amsart}

\usepackage[top=3.75cm, bottom=3cm, left=3cm, right=3cm]{geometry}
\usepackage{amsmath}
\usepackage{amsthm}
\usepackage{amsfonts}
\usepackage{amssymb}
\usepackage{eucal}
\usepackage{bm}
\usepackage[all,cmtip]{xy}
\usepackage{hyperref}
\usepackage{color}

\numberwithin{equation}{section}

\theoremstyle{plain}
\newtheorem{theorem}{Theorem}[section]
\newtheorem{lemma}[theorem]{Lemma}
\newtheorem{proposition}[theorem]{Proposition}
\newtheorem{corollary}[theorem]{Corollary}

\theoremstyle{definition}
\newtheorem{definition}[theorem]{Definition}
\newtheorem{remark}[theorem]{Remark}

\DeclareMathOperator{\Coh}{Coh}

\DeclareMathOperator{\Rep}{Rep}

\newcommand{\Db}{\mathrm{D^b}}

\DeclareMathOperator{\G}{G}

\DeclareMathOperator{\Spec}{Spec}
\DeclareMathOperator{\Hom}{Hom}

\DeclareMathOperator{\Tor}{Tor}
\DeclareMathOperator{\Aut}{Aut}

\newcommand{\one}{\mathbf{1}}
\newcommand{\id}{\mathrm{id}}
\DeclareMathOperator{\Cone}{Cone}
\DeclareMathOperator{\Pic}{Pic}

\newcommand{\cO}{\mathcal{O}}
\newcommand{\cA}{\mathcal{A}}
\newcommand{\cB}{\mathcal{B}}
\newcommand{\cC}{\mathcal{C}}

\newcommand{\cF}{\mathcal{F}}
\newcommand{\cG}{\mathcal{G}}
\newcommand{\cH}{\mathcal{H}}
\newcommand{\cL}{\mathcal{L}}

\newcommand{\cR}{\mathcal{R}}
\newcommand{\cT}{\mathcal{T}}
\newcommand{\cU}{\mathcal{U}}

\newcommand{\rS}{\mathrm{S}}
\newcommand{\rL}{\mathrm{L}}
\newcommand{\rR}{\mathrm{R}}

\newcommand{\Tnsr}[1]{{#1}}

\newcommand{\bC}{\mathbf{k}}
\newcommand{\bZ}{\mathbf{Z}}
\newcommand{\bP}{\mathbf{P}}
\newcommand{\bmu}{\bm{\mu}}
\newcommand{\hbmu}{\widehat{\bmu}}

\newcommand{\sO}{\mathsf{O}}

\begin{document}

\title[Derived categories of cyclic covers and their branch divisors]{Derived categories of cyclic covers\\[1ex]and their branch divisors}

\author{Alexander Kuznetsov}
\address{{\sloppy
\parbox{0.9\textwidth}{
Steklov Mathematical Institute,
8 Gubkin str., Moscow 119991 Russia
\\[5pt]
The Poncelet Laboratory, Independent University of Moscow
\hfill\\[5pt]
National Research University Higher School of Economics, Russian Federation
}\bigskip}}
\email{akuznet@mi.ras.ru \medskip}

\author{Alexander Perry}
\address{Department of Mathematics, Harvard University, Cambridge, MA 02138, USA \medskip}
\email{aperry@math.harvard.edu}

\thanks{A.K. was partially supported by a subsidy granted to the HSE by 
the Government of the Russian Federation for the implementation of the Global 
Competitiveness Program and by RFBR grants 14-01-00416, 15-01-02164, 15-51-50045 and NSh-2998.2014.1.
A.P. was partially supported by NSF GRFP grant DGE1144152, 
and thanks the Laboratory of Algebraic Geometry SU-HSE for its hospitality in November 2013, 
when this work was initiated.}

\date{\today}

\begin{abstract}
Given a variety $Y$ with a rectangular Lefschetz decomposition of its
derived category, we consider a degree $n$ cyclic cover $X \to Y$ ramified over 
a divisor $Z \subset Y$. 
We construct semiorthogonal decompositions of $\Db(X)$ and $\Db(Z)$ with 
distinguished components $\cA_X$ and $\cA_Z$, and prove the equivariant category of $\cA_X$
(with respect to an action of the $n$-th roots of unity) admits a semiorthogonal decomposition
into $n-1$ copies of $\cA_Z$. As examples we consider quartic double solids,
Gushel--Mukai varieties, and cyclic cubic hypersurfaces.
\end{abstract}

\maketitle


\section{Introduction}

Let $Y$ be an algebraic variety with a line bundle $\cO_Y(1)$. 
Assume the bounded derived category of coherent sheaves $\Db(Y)$ is equipped
with a rectangular Lefschetz decomposition with respect to $\cO_Y(1)$. In other words, 
assume an admissible subcategory $\cB \subset \Db(Y)$ is given such that 
\begin{equation*}
\Db(Y) = \langle \cB, \cB(1), \dots, \cB(m-1) \rangle
\end{equation*}
is a semiorthogonal decomposition, where $\cB(t)$ stands for the image of $\cB$ under 
the autoequivalence $\cF \mapsto \cF \otimes \cO_Y(t)$ of $\Db(Y)$.

Choose positive integers $n$ and $d$ such that $nd \le m$. Let $f: X \to Y$ be a 
degree $n$ cyclic cover of $Y$, ramified over a Cartier divisor $Z$ in the linear 
system $|\cO_Y(nd)|$. Let $i: Z \hookrightarrow Y$ be the inclusion. 
Then the derived pullback functor $f^*:\Db(Y) \to \Db(X)$ is fully faithful on the 
subcategory $\cB(t)$ for any~$t$, and the same is true of $i^*: \Db(Y) \to \Db(Z)$
provided $nd < m$. Moreover, denoting by $\cB_X(t) \subset \Db(X)$ and 
$\cB_Z(t) \subset \Db(Z)$ the images of $\cB(t)$, there are semiorthogonal
decompositions
\begin{equation*}
\Db(X) = \langle \cA_X, \cB_X, \cB_X(1), \dots, \cB_X(m-(n-1)d-1) \rangle
\end{equation*}
and
\begin{equation*}
\Db(Z) = \langle \cA_Z, \cB_Z, \cB_Z(1), \dots, \cB_Z(m-nd-1) \rangle.
\end{equation*}
Here $\cA_X$ and $\cA_Z$ are defined as the right orthogonal categories to the copies
of $\cB$ appearing in the semiorthogonal decompositions. The goal of this paper is to 
demonstrate a relation between $\cA_X$ and $\cA_Z$.

Explicitly, we consider the action of the group of $n$-th roots of unity $\bmu_n$ on
$X$ by automorphisms over $Y$. This action preserves $\cA_X$ since it preserves 
each category $\cB_X(t)$ in the above semiorthogonal decomposition. We denote by 
$\cA_X^{\bmu_n}$ the category of $\bmu_n$-equivariant objects of~$\cA_X$. 
Our main result can be stated as follows. 

\begin{theorem}\label{main-theorem}
Let $Y$ be a variety with a rectangular Lefschetz decomposition of length $m$ with respect 
to a line bundle $\cO_Y(1)$.
Let $X \to Y$ be a degree $n$ cyclic cover of $Y$, ramified over a Cartier divisor $Z$ in 
the linear system $|\cO_Y(nd)|$, where $nd \leq m$. 
Then for $\cA_Z$ and $\cA_X$ defined as above,
there
are fully faithful functors $\Phi_k: \cA_Z \to \cA_{X}^{\bmu_n}$, $0 \leq k \leq n-2$, 
such that there is a semiorthogonal decomposition
\begin{equation*}
\cA_X^{\bmu_n} = \langle \Phi_0(\cA_Z), \Phi_1(\cA_Z), \dots, \Phi_{n-2}(\cA_Z) \rangle.
\end{equation*}
\end{theorem}

See Theorem~\ref{main-theorem-precise} for an explicit formula for the functor  
$\Phi_k$. 

\begin{remark}
\label{remark-stacks}
Theorem~\ref{main-theorem} should hold more generally (with the same proof) 
if $Y$ is a Deligne--Mumford stack (e.g.\ a weighted projective space). We have chosen not to work in this generality 
since for some of the results we need, the references only treat the case of varieties.
\end{remark}

If $n = 2$, the theorem gives an equivalence $\cA_X^{\bmu_2} \simeq \cA_Z$. 
By a result of Elagin~\cite{elagin2014equivariant}, we deduce in this case a ``dual'' equivalence.

\begin{corollary}\label{main-corollary}
If $n = 2$ in the setup of Theorem~\ref{main-theorem}, then 
there is an action of $\bZ/2$ on $\cA_Z$ such that $\cA_X \simeq \cA_Z^{\bZ/2}$.
\end{corollary}

In Proposition~\ref{proposition-tau}, we describe this $\bZ/2$-action explicitly in terms of 
a natural ``rotation functor'' associated to $\cA_Z$.

If $n>2$, there is still a description of $\cA_X$ in terms of $\cA_Z$, 
but it is more complicated. In this case, $\cA_X$ can be recovered as the 
$\bZ/n$-equivariant category of a ``gluing'' of $n-1$ copies of $\cA_Z$. 
For $n = 3$, we speculate about a way to make this reconstruction result 
more explicit, in terms of~$\cA_Z$ and its associated ``rotation functor.'' 

Components of semiorthogonal decompositions of algebraic varieties are in some cases related
to their Hodge structures
(see~\cite[Section~2]{addington-thomas} for the case of a cubic fourfold).
From this point of view, the above reconstruction results are surprising, 
since typically there is no direct description of the Hodge structure of a cyclic cover 
in terms of the Hodge structure of its branch divisor.

We emphasize that in our results we do not assume the varieties involved 
to be smooth. Note that even if $Y$ is smooth (as will often be true in applications), 
the cover $X$ and its branch divisor $Z$ can easily be singular. 
We also never assume the varieties involved to be proper, as all the results are local 
with respect to $Y$.

We apply our main theorem to three cases: 
quartic double solids,
Fano varieties of Picard number 1, degree 10,
and coindex 3 (so-called Gushel--Mukai varieties), and 
cyclic cubic hypersurfaces. 
In particular, we give interesting examples of 
K3 categories equipped with a group action, and 
describe their equivariant category. 

\subsection*{Organization of the paper} 
In Sections~\ref{section-triangulated-categories} and~\ref{section-group-actions}, 
we discuss background material on semiorthogonal decompositions and group 
actions on categories. 
In Section~\ref{section-cyclic-cover}, we give a 
semiorthogonal decomposition of the equivariant derived category of a cyclic cover. 
In Section~\ref{section-induced-sods}, we establish the semiorthogonal 
decompositions of $\Db(X)$ and $\Db(Z)$ mentioned above. 
In Section~\ref{section-proof}, we prove Theorem~\ref{main-theorem}.
In Section~\ref{section-reconstruction}, we introduce the rotation functors and 
prove the reconstruction results stated above.
Finally, in Section~\ref{section-applications} we apply our results to several 
examples.

\subsection*{Conventions} 
We work over an algebraically closed field $\bC$ of characteristic coprime to $n$. 
Varieties will be assumed integral and quasi-projective over $\bC$.

\subsection*{Acknowledgements}
A.K. is grateful to Alexey Elagin for his clarifications concerning equivariant categories.
A.P. thanks Joe Harris and Johan de Jong for useful conversations related to this work.


\section{Preliminaries on triangulated categories}
\label{section-triangulated-categories}
In this paper, triangulated categories are $\bC$-linear 
and functors between them are $\bC$-linear and exact.

\subsection{Derived categories of varieties}
We use the following notation:
For a variety $X$, we denote by $\Db(X)$ the bounded derived category of coherent 
sheaves on $X$, regarded as a triangulated category. We refer to $\Db(X)$ simply 
as the derived category of $X$. For a morphism of varieties $f: X \to Y$, we write  
$f_*: \Db(X) \to \Db(Y)$
for the derived pushforward (provided $f$ is proper), and 
$f^*: \Db(Y) \to \Db(X)$
for the derived pullback (provided $f$ has finite $\Tor$-dimension). 
For $\cF, \cG \in \Db(X)$, we write $\cF \otimes \cG \in \Db(X)$ for the derived tensor product.

\subsection{Semiorthogonal decompositions}
\label{subsection-sod-review}
We recall some well-known facts about semiorthogonal decompositions.
We suggest the reader consult~\cite{bondal} and~\cite{bondal-kapranov} for more details, 
or~\cite{kuznetsov2009hochschild} 
for a short review of results.

\begin{definition}
\label{definition-sod}
Let $\cT$ be a triangulated category. 
A \emph{semiorthogonal decomposition} 
\begin{equation*}
\cT = \langle \cA_1, \dots, \cA_n \rangle
\end{equation*}
is a sequence of full triangulated subcategories $\cA_1, \dots, \cA_n$ of $\cT$ 
--- called the \emph{components} of the decomposition --- such that: 
\begin{enumerate}
\item \label{definition-sod-1} 
$\Hom(\cF, \cG) = 0$ for all $\cF \in \cA_i, \cG \in \cA_j$ if $i>j$.
\item \label{definition-sod-2}
For any $\cF \in \cT$, there is a sequence of morphisms
\begin{equation*}
0 = \cF_n \to \cF_{n-1} \to \cdots \to \cF_1 \to \cF_0 = \cF,
\end{equation*}
such that $\Cone(\cF_i \to \cF_{i-1}) \in \cA_i$.
\end{enumerate}
If only condition~\eqref{definition-sod-1} is satisfied, 
we say the sequence $\cA_1, \dots, \cA_n$ is \emph{semiorthogonal}.
\end{definition}

\begin{remark}
\label{remark-projection-functors}
Condition~\eqref{definition-sod-1} of the definition implies the 
``filtration'' in~\eqref{definition-sod-2} and its ``factors'' are unique and functorial.
The functors 
\begin{equation*}
\cT \to \cA_i, \quad \cF \mapsto \Cone(\cF_{i} \to \cF_{i-1})
\end{equation*}
are called the \emph{projection functors} of the semiorthogonal decomposition. 
We call the object $\Cone(\cF_i \to \cF_{i-1})$ the \emph{component} of $\cF$ in 
$\cA_i$.
\end{remark}

A full triangulated subcategory $\cA \subset \cT$ is called \emph{right admissible} if the 
embedding functor ${\alpha:\cA \to \cT}$ has a right adjoint $\alpha^!:\cT \to \cA$, 
\emph{left admissible} if $\alpha$ has a left adjoint $\alpha^*:\cT \to \cA$, and 
\emph{admissible} if it is both right and left admissible.

If a semiorthogonal decomposition $\cT = \langle \cA,\cB \rangle$ is given,  
then $\cA$ is left admissible and $\cB$ is right admissible.
Vice versa, if $\cA \subset \cT$ is right admissible then there is a semiorthogonal 
decomposition 
\begin{equation}
\label{sod-right-admissible}
\cT = \langle \cA^\perp, \cA \rangle,
\end{equation} 
and if $\cA$ is left admissible then there is a semiorthogonal decomposition 
\begin{equation}
\label{sod-left-admissible}
\cT = \langle \cA, {}^\perp\cA \rangle.
\end{equation}
Here $\cA^{\perp}$ and ${}^{\perp}\cA$ denote the \emph{right} and \emph{left orthogonal} 
categories to $\cA$, i.e. the full subcategories of $\cT$ given by
\begin{align*}
\cA^{\perp} & = \left \{ \cF \in \cT ~ | ~ \Hom(\cG, \cF) = 0 \text{ for all } \cG \in \cA \right \},  \\
{}^{\perp}\cA & =  \left \{ \cF \in \cT ~ | ~ \Hom(\cF, \cG) = 0 \text{ for all } \cG \in \cA \right \}.
\end{align*}

\subsection{Mutations}
\label{subsection-mutations}
If $\cA \subset \cT$ is right admissible, then for any object $\cF \in \cT$ there is a 
distinguished triangle
\begin{equation*}
\alpha\alpha^!(\cF) \to \cF \to \rL_\cA(\cF),
\end{equation*}
where $\rL_\cA(\cF)$ is defined as the cone of the counit morphism. 
The first and last vertices give the projections of $\cF$ to $\cA$ and $\cA^{\perp}$ 
with respect to~\eqref{sod-right-admissible}. 
Similarly, if $\cA \subset \cT$ is left 
admissible, there is a distinguished triangle 
\begin{equation*}
\rR_\cA(\cF) \to \cF \to \alpha\alpha^*(\cF),
\end{equation*}
with vertices the projections of $\cF$ to ${}^{\perp}\cA$ and $\cA$ with respect 
to~\eqref{sod-left-admissible}. Since the projections to $\cA^{\perp}$ and ${}^{\perp}\cA$ 
are functorial (Remark~\ref{remark-projection-functors}), the above prescriptions 
define functors 
\begin{equation*}
\rL_{\cA} : \cT \to \cT \qquad \text{and} \qquad
\rR_{\cA} : \cT \to \cT,
\end{equation*} 
called the~\emph{left} and \emph{right mutation functors} of $\cA \subset \cT$.

If $\cA \subset \cT$ is admissible, both functors $\rL_{\cA}$ and $\rR_{\cA}$ are defined. 
In this case, they both annihilate $\cA$ and the restrictions
\begin{equation*}
\rL_{\cA}|_{{}^{\perp}\cA} : {}^{\perp}\cA \to \cA^{\perp}
\qquad \text{and} \qquad
\rR_{\cA}|_{\cA^{\perp}}  : \cA^{\perp} \to {}^{\perp}\cA
\end{equation*}
are mutually inverse equivalences (\cite[Lemma~1.9]{bondal-kapranov}). 
Furthermore, if $\cA_1, \dots, \cA_n$ is a semiorthogonal sequence of admissible subcategories 
of $\cT$, then $\langle \cA_1, \dots, \cA_n \rangle \subset \cT$ is admissible and
\begin{align}
\label{left-mutation-composition}
\rL_{\langle \cA_1, \dots, \cA_n \rangle} & \cong \rL_{\cA_1} \circ \rL_{\cA_2}
\circ \cdots \circ \rL_{\cA_n}, \\
\label{right-mutation-composition}
\rR_{\langle \cA_1, \dots, \cA_n \rangle} & \cong \rR_{\cA_n} \circ \rR_{\cA_{n-1}} 
\circ \cdots \circ \rR_{\cA_1}.
\end{align}

We are interested in mutation functors because they act on semiorthogonal decompositions. 
The next result follows easily from~\cite[Lemma 1.9]{bondal-kapranov}, where the case $n = 2$ is considered.

\begin{proposition}
\label{prop-mutations}
Let $\cT = \langle \cA_1,\dots,\cA_n \rangle$ be a semiorthogonal decomposition 
with admissible components. 
Then for $1 \leq i \leq n-1$ there is a semiorthogonal decomposition
\begin{equation*}
\cT  = \langle \cA_1, \dots, \cA_{i-1}, \rL_{\cA_i}(\cA_{i+1}), \cA_i, \cA_{i+2}, \dots, \cA_n \rangle,
\end{equation*}
and for $2 \leq i \leq n$ there is a semiorthogonal decomposition
\begin{equation*}
\cT  = \langle \cA_1, \dots, \cA_{i-2}, \cA_i, \rR_{\cA_i}(\cA_{i-1}), \cA_{i+1}, \dots, \cA_n \rangle.
\end{equation*}
\end{proposition}

The following observation is useful for computing mutations. 
Let $\cA \subset \cT$ be an admissible subcategory. 
For any $\cF \in \cT$, to compute $\rL_{\cA}(\cF)$ it suffices to construct 
a distinguished triangle
\begin{equation}
\label{left-mutation-triangle}
\cG \to \cF \to \cG'
\end{equation}
with $\cG \in \cA$ and $\cG' \in \cA^{\perp}$; indeed, then $\rL_{\cA}(\cF) \cong \cG'$ 
by Remark~\ref{remark-projection-functors}. Similarly, if we 
construct a distinguished triangle
\begin{equation}
\label{right-mutation-triangle}
\cH' \to \cF \to \cH
\end{equation}
with $\cH \in \cA$ and $\cH' \in {}^{\perp}\cA$, then $\rR_{\cA}(\cF) \cong \cH'$. 
We call a triangle as in~\eqref{left-mutation-triangle} or~\eqref{right-mutation-triangle} 
a \emph{mutation triangle}. 
The following two lemmas can easily be proved using the description of 
mutation functors in terms of mutation triangles.

\begin{lemma}
\label{mutation-permutation}
Let $\cT = \langle \cA_1,\dots,\cA_n \rangle$ be a semiorthogonal decomposition 
with admissible components. 
Assume for some $i$ the components $\cA_i$ and $\cA_{i+1}$ are completely orthogonal, i.e.\
$\Hom(\cF,\cG) = \Hom(\cG,\cF) = 0$ for all $\cF \in \cA_i$, $\cG \in \cA_{i+1}$. 
Then $\rL_{\cA_i}(\cG) = \cG$ for any $\cG \in \cA_{i+1}$, 
and $\rR_{\cA_{i+1}}(\cF) = \cF$ for any 
$\cF \in \cA_i$. In particular, 
\begin{equation*}
\cT = \langle \cA_1,\dots,\cA_{i-1},\cA_{i+1},\cA_i,\cA_{i+2},\dots,\cA_n \rangle
\end{equation*}
is a semiorthogonal decomposition.
\end{lemma}

\begin{lemma}
\label{lemma-mutation-autoequivalence}
Let $F: \cT_1 \to \cT_2$ be an equivalence of triangulated categories. 
Let $\cA \subset \cT_1$ be an admissible subcategory. Then 
$F \circ \rL_{\cA} \cong \rL_{F(\cA)} \circ F$ 
and
$F \circ \rR_{\cA} \cong \rR_{F(\cA)} \circ F$.
\end{lemma}


\section{Group actions on triangulated categories}
\label{section-group-actions}
In this section we discuss (finite) group actions on categories.
After recalling the definition of the equivariant category of 
a group action on an arbitrary category, we focus on the 
triangulated case. In Section~\ref{subsection-triangulated-equivariant-categories}, 
we describe several situations where the equivariant category is naturally triangulated. 
We also state a result of Elagin, which gives a semiorthogonal decomposition of the 
equivariant derived category of a variety induced by a decomposition of the non-equivariant 
derived category. In Section~\ref{subsection-trivial-actions}, we give a semiorthogonal 
decomposition of the equivariant category of a trivial action. Finally, in 
Section~\ref{subsection-equivariant-categories-varieties} we summarize some
facts about the equivariant derived categories of varieties.

\subsection{Equivariant categories}
\label{subsection-equivariant-categories}
Suppose given a (right) action of a finite group $G$ on a category $\cC$. 
In other words, suppose given:
\begin{itemize}
\item For every $g \in G$, an autoequivalence $g^*: \cC \to \cC$.
\item For every $g, h \in G$, an isomorphism of functors 
$c_{g,h}: (gh)^* \xrightarrow{\sim} h^* \circ g^*$, such that the diagram
\begin{equation*}
\xymatrix{
(fgh)^* \ar[d]_{c_{f,gh}} \ar[rr]^{c_{fg,h}} & & h^* \circ (fg)^* \ar[d]^{h^*  c_{f,g}} \\
(gh)^* \circ f^* \ar[rr]^{c_{g,h} f^*} & & h^* \circ g^* \circ f^*
}
\end{equation*}
commutes for all $f,g,h \in G$.
\end{itemize}

\begin{definition}
A \emph{$G$-equivariant object} of $\cC$ is a pair $(\cF, \phi)$,
where $\cF$ is an object of $\cC$ and $\phi$ is a collection of isomorphisms 
$\phi_{g}: \cF \xrightarrow{\sim} g^*(\cF)$ for $g \in G$, such that the diagram
\[
\xymatrix{
\cF \ar[rr]^{\phi_{h}} \ar[drrrr]_{\phi_{gh}} & & h^*(\cF) \ar[rr]^{h^*(\phi_{g})} 
& & h^*(g^*(\cF)) \\
& &  & & (gh)^*(\cF) \ar[u]_{c_{g,h}(\cF)} 
}
\]
commutes for all $g, h \in G$. 
The equivariant structure $\phi$ is often omitted from the notation. 
\end{definition}

The \emph{equivariant category $\cC^G$} of $\cC$ is the category of 
$G$-equivariant objects of $\cC$, with the obvious morphisms. If $\cC$ is additive
and $G$ acts by additive autoequivalences, then $\cC^G$ is additive.

\subsection{Triangulated equivariant categories}
\label{subsection-triangulated-equivariant-categories}
Now assume a finite group $G$ acts on a triangulated category $\cT$ by exact autoequivalences. 
We will always assume in this situation that the order of $G$ is coprime to the characteristic 
of the base field $\bC$.
The category $\cT^G$ is additive and comes with a natural shift functor 
and a class of distinguished triangles. Namely, the shift functor on $\cT^G$ 
is given by $(\cF,\phi) \mapsto (\cF[1], \phi[1])$, 
and a triangle in $\cT^G$ is distinguished if and only if the underlying triangle 
in $\cT$ is distinguished. In general, $\cT$ being triangulated is not 
sufficient to guarantee this defines a triangulated structure on $\cT^G$ 
(see~\cite{elagin2011cohomological} for a more detailed discussion).
In case this \emph{does} define a triangulated structure, 
we will simply say ``$\cT^G$~is triangulated.''

The category  $\cT^G$ is triangulated in certain situations, 
typically when it can be identified with a category which is a priori triangulated. 
First, this holds if $\cT = \Db(X)$ for a variety~$X$ and 
$G$ acts via automorphisms of $X$. Indeed, in this case $\cT^G$ is equivalent to 
the bounded derived category of $\Coh(X)^G$ (see \cite[Theorem~9.6]{elagin2011cohomological}).
Second, $\cT^G$ is triangulated if $\cT$ is a semiorthogonal component of 
$\Db(X)$ and $G$ acts via automorphisms of $X$ that preserve $\cT$. 
In fact, in this case the following theorem shows $\cT^G$ is 
a semiorthogonal component of $\Db(X)^G$, hence in particular triangulated.

\begin{theorem}[\cite{elagin2012descent,elagin2014equivariant}]
\label{theorem-group-action-sod}
Let $X$ be a quasi-projective variety with an action of a finite group~$G$. Assume
$\Db(X) = \langle \cA_1, \dots, \cA_n \rangle$ is a semiorthogonal decomposition which 
is preserved by $G$, i.e. each $\cA_i$ is preserved by the action of $G$. 
Then there is a semiorthogonal decomposition
\begin{equation}
\label{equation-g-sod}
\Db(X)^{G} = \langle \cA_1^G, \dots, \cA_n^G \rangle.
\end{equation}
\end{theorem}

\begin{proof}
The equivariant category $\Db(X)^G$ comes with an induced semiorthogonal decomposition 
by~\cite[Theorem~6.3]{elagin2012descent}, and its components are equivalent
to the equivariant categories~$\cA_i^G$ by \cite[Proposition~3.10]{elagin2014equivariant}.
\end{proof}

\subsection{Trivial actions}
\label{subsection-trivial-actions}
We say the action of $G$ on a category $\cC$ is \emph{trivial} if for each $g \in G$
an isomorphism of functors $\tau_g: \id \xrightarrow{\sim} g^*$ is given, such that 
\begin{equation*}
c_{g,h} \circ \tau_{gh} = h^*\tau_g \circ \tau_h
\end{equation*}
for all $g, h \in G$.

\begin{proposition}
\label{proposition-trivial-action-sod}
Let $\cT$ be a triangulated category with a trivial action of a finite group~$G$.
If $\cT^G$ is also triangulated, then there is a completely orthogonal decomposition
\begin{equation}
\label{equation-trivial-action-sod}
\cT^G = \langle \cT \otimes V_0, \cT \otimes V_1, \dots, \cT \otimes V_n \rangle,
\end{equation}
where $V_0, \dots, V_n$ is a list of the finite-dimensional irreducible representations 
of $G$.
\end{proposition}
\begin{proof}
Since the action of $G$ is trivial, a $G$-equivariant object of $\cT$
is the same as an object $\cF \in \cT$ with a representation $G \to \Aut(\cF)$.
In particular, for any $\cF \in \cT$ and $V \in \Rep(G)$ 
(the category of finite-dimensional representations of $G$), 
the tensor product $\cF \otimes V$ is a $G$-equivariant object of $\cT$.
Moreover, given $\cF, \cF' \in \cT$ and $V, V' \in \Rep(G)$, then
\begin{equation*}
\Hom_{\cT^G}(\cF \otimes V, \cF' \otimes V') \cong 
\Hom_{\cT}(\cF,\cF') \otimes \Hom_{\Rep(G)}(V,V').
\end{equation*}
Hence the functors $\cT \to \cT^G$ given by $\cF \mapsto \cF \otimes V_k$ 
are fully faithful, 
and the essential images $\cT \otimes V_k$ of these functors are completely orthogonal.
Finally, if $\cF \in \cT^G$ then
\begin{equation*}
\cF = \cF \otimes_{\bC[G]} \bC[G] =
\cF \otimes_{\bC[G]} \left( \oplus_{k=0}^n V_k^\vee\otimes V_k \right) =
\bigoplus_{k=0}^n (\cF \otimes_{\bC[G]}  V_k^\vee) \otimes V_k,
\end{equation*}
which proves the categories $\cT \otimes V_k$ generate $\cT^G$.
\end{proof}

We will apply Proposition~\ref{proposition-trivial-action-sod} in the following situation: 
$G$ acts on a variety $X$, $\cT$ is a semiorthogonal component of $\Db(X)$ preserved by 
$G$ (so that $\cT^G$ is triangulated by Theorem~\ref{theorem-group-action-sod}), 
and $G$ induces a trivial action on $\cT$.
In this situation, we define functors
\begin{align}
\label{iota_k} \iota_k & : \cT \to \cT^G, \quad \cF \mapsto \cF \otimes V_k, \\
\label{pi_k} \pi_k & : \cT^G \to \cT, \quad  \cF \mapsto \cF \otimes_{\bC[G]} V_k^\vee.
\end{align}
The proof of Proposition~\ref{proposition-trivial-action-sod}
shows the $\pi_k$ are the projection functors for the
semiorthogonal decomposition~\eqref{equation-trivial-action-sod}.  
Since this decomposition is completely orthogonal, the functors $\pi_k$ are 
simultaneously left and right adjoint to $\iota_k$, and we have 
\begin{equation*}
\pi_k \circ \iota_\ell = 
\begin{cases}
\id_{\cT} & \text{if $k = \ell$,}\\
0 & \text{if $k \neq \ell$.}
\end{cases}
\end{equation*}

\subsection{Equivariant derived categories of varieties}
\label{subsection-equivariant-categories-varieties}
The usual functor formalism for categories of sheaves extends directly 
to the equivariant setting. We summarize the relevant facts here, 
referring to~\cite{mckay-correspondence} for a more detailed exposition. 
Let $G$ be a finite group and $f: X \to Y$ a $G$-equivariant morphism of varieties 
with $G$-actions. Then if $f$ is proper there is a derived pushforward functor
\begin{equation*}
f_* : \Db(X)^G \to \Db(Y)^G,
\end{equation*}
and if $f$ is of finite $\Tor$-dimension there is a derived pullback functor
\begin{equation*}
f^*: \Db(Y)^G \to \Db(X)^G.
\end{equation*}
When both functors are defined, $f^*$ is left adjoint to $f_*$. Similarly, there is 
a derived tensor product for equivariant complexes $\cF, \cG \in \Db(X)^G$, 
which we denote by $\cF \otimes \cG$. 
These functors satisfy the same relations
as in the non-equivariant setting, e.g. the projection formula 
\begin{equation}\label{projection-formula}
f_*(\cF \otimes f^*\cG) \cong f_*(\cF) \otimes \cG
\end{equation} 
holds for any $\cF \in \Db(X)^G$ and $\cG \in \Db(Y)^G$. 
Moreover, there is an equivariant Grothendieck duality, 
providing a right adjoint $f^!: \Db(Y)^G \to \Db(X)^G$ to $f_*$ when $f$ is proper. 

Finally, assume $f: X \to Y$ is an inclusion of a Cartier divisor, defined by 
an invariant section of a $G$-equivariant line bundle $\cL$ on $Y$. 
In the non-equivariant situation, i.e. when $G$ is trivial, we have 
the standard formula
\begin{equation}
\label{shriek-divisor}
f^!(\cF) = f^*(\cF) \otimes f^*(\cL)[-1].
\end{equation}
Moreover, by~\cite[Lemma~3.3]{sod-bondal-orlov}, 
the canonical adjunction morphism $f^*(f_*\cF) \to \cF$ extends to 
the standard distinguished triangle
\begin{equation}
\label{equivariant-divisor-triangle}
\cF \otimes f^*(\cL)^{-1}[1] \to f^*(f_*\cF) \to \cF. 
\end{equation}
Both~\eqref{shriek-divisor} and~\eqref{equivariant-divisor-triangle} 
continue to hold when $G$ is nontrivial. 


\section{The equivariant derived category of a cyclic cover}
\label{section-cyclic-cover}

\subsection{Setup and notation}
Let $Y$ be an algebraic variety and $\cL$ a line bundle on $Y$. Suppose $Z$ is a Cartier 
divisor in $Y$ defined by a section of $\cL^{n}$. Let $f: X \to Y$
be the degree $n$ cyclic cover of $Y$ ramified over~$Z$, i.e.
\begin{equation*}
X = \Spec_Y(\cR_Y),
\qquad
\cR_Y := \cO_Y \oplus \cL^{-1} \oplus \dots \oplus \cL^{-(n-1)},
\end{equation*}
where the algebra structure on $\cR_Y$ is given by the 
map $\cL^{-n} \to \cO_Y$ corresponding to the
divisor $Z$. 

Let $\bmu_n$ denote the group of $n$-th roots of unity. 
Its dual group $\hbmu_{n}$ (the group of characters) is a cyclic group. 
We identify $\hbmu_{n}$ with $\bZ/n$ by choosing a primitive character 
$\chi:\bmu_n \to \bC^\times$ and associating $k \in \bZ/n$ to $\chi^k \in \hbmu_{n}$.
Thus all irreducible representations of $\bmu_n$ are given by
\begin{equation*}
V_0 = \one, \ V_1 = \chi, \ \dots, \ V_{n-1} = \chi^{n-1},
\end{equation*}
and are indexed by $\bZ/n$.

We equip $Y$ with the trivial $\bmu_n$-action. The group $\bmu_n$ acts on 
the sheaf $\cR_Y$ via the character $\chi^k$ on the summand $\cL^{-k}$, 
so that as an equivariant sheaf it can be written as
\begin{equation}\label{equation-ry}
\cR_Y = (\cO_Y \otimes \one) \oplus (\cL^{-1} \otimes \chi) \oplus \dots \oplus 
(\cL^{-(n-1)} \otimes \chi^{n-1}).
\end{equation} 
This induces an action of $\bmu_n$ on $X = \Spec_Y(\cR_Y)$.

The morphism $f: X \to Y$ is $\bmu_n$-equivariant with respect to the above actions.
Moreover, it is proper and flat. 
For each $k \in \bZ/n$, we define the functors 
$f_k^*$, $f_{k*}$, $f_k^!$ as the compositions
\begin{align*}
f_k^* & : \Db(Y) \xrightarrow{\ \iota_k\ } \Db(Y)^{\bmu_n} \xrightarrow{\ f^*\ } \Db(X)^{\bmu_n} , \\
f_{k*} & : \Db(X)^{\bmu_n} \xrightarrow{\ f_*\ } \Db(Y)^{\bmu_n} \xrightarrow{\ \pi_k\ }  \Db(Y), \\
f_k^! & : \Db(Y) \xrightarrow{\ \iota_k \ } \Db(Y)^{\bmu_n} \xrightarrow{\ f^! \ } \Db(X)^{\bmu_n},
\end{align*}
where $\iota_k$ and $\pi_k$ are given by~\eqref{iota_k} and~\eqref{pi_k}.
Then $f_{k}^*$ is the left adjoint of $f_{k*}$, and $f_k^!$ is the right adjoint of $f_{k*}$.

Let $\cL_X$ and $\cL_Z$ be the pullbacks of $\cL$ to $X$ and $Z$ respectively. 
By~\eqref{equation-ry} the line bundle $\cL_X \otimes \chi^{-1}$ has a 
$\bmu_n$-invariant section on $X$. The zero locus of this section is the 
ramification divisor of the cover $X \to Y$, which can be identified 
with $Z$. Denoting by $i: Z \hookrightarrow Y$ and $j: Z \hookrightarrow X$ 
the embeddings of $Z$ as the branch divisor and ramification divisor, we have 
a commutative diagram
\begin{equation}\label{diagram-xyz}
\vcenter{\xymatrix{
& X \ar[d]^f  \\ 
Z \ar[ur]^j \ar[r]^i & Y
}}
\end{equation}
which is $\bmu_n$-equivariant if $Z$ is equipped with the trivial $\bmu_n$-action and the 
actions on $X$ and $Y$ are as above. 
The embeddings $i$ and $j$ are proper and have $\Tor$-dimension $1$. 
For each $k \in \bZ/n$, we define the functors $j_k^*$, $j_{k*}$, $j_k^!$ as the compositions
\begin{align*}
j_k^* & : \Db(X)^{\bmu_n} \xrightarrow{\ j^*\ } \Db(Z)^{\bmu_n} \xrightarrow{\ \pi_k\ } \Db(Z), \\
j_{k*} & : \Db(Z) \xrightarrow{\ \iota_k\ } \Db(Z)^{\bmu_n} \xrightarrow{\ j_*\ } \Db(X)^{\bmu_n}, \\
j_k^! & : \Db(X)^{\bmu_n} \xrightarrow{\ j^! \ } \Db(Z)^{\bmu_n} \xrightarrow{\ \pi_k \ } \Db(Z). 
\end{align*}
Again, $j_{k}^*$ is the left adjoint of $j_{k*}$, and $j_k^!$ is the right adjoint of $j_{k*}$.
Note that since $Z$ is the zero locus of an invariant section of $\cL_X\otimes \chi^{-1}$, 
we have an equivariant exact sequence
\begin{equation}\label{equation-resolution-z}
0 \to \cL^{-1}_X \otimes \chi \to \cO_X \otimes \one \to j_{0*}\cO_Z \to 0.
\end{equation} 

\subsection{Semiorthogonal decompositions}

The action of $\bmu_n$ on $Z$ is trivial, hence by Proposition~\ref{proposition-trivial-action-sod}
we have a semiorthogonal decomposition
\begin{equation}\label{equation-dbz-equivariant}
\Db(Z)^{\bmu_n} = \langle \iota_0(\Db(Z)), \iota_1(\Db(Z)), \dots, \iota_{n-1}(\Db(Z)) \rangle.
\end{equation} 
This is similar to Orlov's decomposition of the derived category of the projectivization of a vector bundle
(the twists by characters $\chi^k$ in the definition of functors $\iota_k$ are analogous to the twists
by line bundles $\cO(k)$ in Orlov's decomposition).
On the other hand, the $\bmu_n$-equivariant derived category of the cyclic cover $X$ 
has a semiorthogonal decomposition which is similar to Orlov's decomposition
of the derived category of a blowup
(the equivariant derived category of the cyclic cover is assembled from the derived category of the base and a part
of the equivariant derived category of the branch divisor in the same way as the derived category of a blowup is assembled from the derived
category of the base and a part of the derived category of the exceptional divisor).

\begin{theorem}
\label{theorem-cyclic-cover-sod}
For each $k \in \bZ/n$ the functors
\begin{align*}
f_k^*  & : \Db(Y) \to \Db(X)^{\bmu_n}, \\
j_{k*} & : \Db(Z) \to \Db(X)^{\bmu_n},
\end{align*}
are fully faithful. Moreover, for $k, \ell \in \bZ/n$ we have:
\begin{align}
(j_{k*}\Db(Z),j_{\ell*}\Db(Z)) & \text{ is semiorthogonal if $k \neq \ell, \ell+1$},
\label{sojj}\\
(j_{k*}\Db(Z),f_{\ell}^*\Db(Y)) 
& \text{ is semiorthogonal if $k \neq \ell$},
\label{sojf}\\
(f_{\ell}^*\Db(Y),j_{k*}\Db(Z))
& \text{ is semiorthogonal if $k \neq \ell-1$}.
\label{sofj}
\end{align}
Finally, for each $k \in \bZ/n$ there is a semiorthogonal decomposition
\begin{multline}
\label{equation-cyclic-cover-sod}
\Db(X)^{\bmu_n} = \langle j_{k+1*}\Db(Z), j_{k+2*}\Db(Z), \dots, j_{n-1*}\Db(Z), \\
f_0^* \Db(Y), j_{0*} \Db(Z), j_{1*}\Db(Z), \dots, j_{k-1*}\Db(Z) \rangle.
\end{multline} 
\end{theorem}

This decomposition is well-known to the experts.
The case $n = 2$ was proved by Collins and Polishchuk in~\cite{gluing-stability-conditions}. 
The general case follows from a result of Ishii and Ueda \cite[Theorem~1.6]{ishii-ueda}, 
which gives a semiorthogonal decomposition of the derived category of a root stack. 
For completeness, we provide a proof here, which is close to Orlov's proof of the 
decomposition of the derived category of a blowup.

\begin{proof}
The functor $f_k^*$ is fully faithful since the composition with its right 
adjoint~$f_{k*}$ satisfies $f_{k*} \circ f_k^* = \id$. 
Indeed, by the projection formula~\eqref{projection-formula} and~\eqref{equation-ry}, we have
\begin{equation*}
f_{k*}(f_k^*\cF) = \pi_k(f_*(f^*(\iota_k\cF))) = \pi_k(\iota_k(\cF) \otimes \cR_Y) =
\pi_k((\cF \otimes \chi^k) \otimes \cR_Y) = \cF.
\end{equation*}
Similarly, to prove $j_{k*}$ is fully faithful we show $j_k^*\circ j_{k*} = \id$. 
For this we take $\cF \in \Db(Z)^{\bmu_n}$ and consider the distinguished triangle
\begin{equation*}
 \cF \otimes \cL_Z^{-1} \otimes \chi[1] \to j^*(j_*\cF) \to \cF,
\end{equation*}
see~\eqref{equivariant-divisor-triangle}.
If $\cF = \iota_k(\cG)$ for $\cG \in \Db(Z)$, the last vertex of this triangle is in 
the component $\Db(Z)\otimes\chi^k$ of $\Db(Z)^{\bmu_n}$, while 
the first is in $\Db(Z)\otimes\chi^{k+1}$. Applying 
$\pi_k$ we deduce $j_k^*(j_{k*}\cG) = \cG$, so that $j_{k*}$ is fully 
faithful. 

To prove~\eqref{sojj}, by adjunction we must show $j_{k}^* \circ j_{\ell*} = 0$
for $k \neq \ell, \ell+1$. This follows by the same argument used to show 
$j_{k}^* \circ j_{k*} = \id$ above.

To prove~\eqref{sojf}, by adjunction we must show $j_{k}^*\circ f_{\ell}^* = 0$ for 
$k \neq \ell$. Since $j_k^* = \pi_k \circ j^*$, this is immediate from the fact that 
the image of $j^* \circ f_{\ell}^*$ lies in the component $\Db(Z) \otimes \chi^{\ell}$ 
of $\Db(Z)^{\bmu_n}$.

To prove~\eqref{sofj}, by adjunction we must show $j_k^! \circ f_{\ell}^* = 0$ for 
$k \neq \ell-1$. If $\cF \in \Db(X)^{\bmu_n}$ then 
\begin{equation*}
j^!(\cF) = j^*(\cF) \otimes j^*(\cL_X \otimes \chi^{-1})[-1] = j^*(\cF) \otimes \cL_Z \otimes \chi^{-1}[-1].
\end{equation*}
So for $\cG \in \Db(Y)$, we have
\begin{equation*}
j^!(f_{\ell}^*(\cG)) = j^*f^*(\cG \otimes \chi^{\ell}) \otimes \cL_Z \otimes \chi^{-1}[-1]
= i^*(\cG) \otimes \cL_Z \otimes \chi^{\ell-1}[-1].
\end{equation*}
Now applying $\pi_k$ proves the required vanishing.

Finally, we establish the semiorthogonal decomposition~\eqref{equation-cyclic-cover-sod}. 
By~\eqref{sojj}, \eqref{sojf}, and~\eqref{sofj}, the components of the claimed 
decomposition are indeed semiorthogonal, so we must show the category~$\cT$
they generate is all of $\Db(X)^{\bmu_n}$. 

First we claim $j_{k*}\Db(Z) \subset \cT$. For this we note that
if $\cG \in \Db(Z)$, then~\eqref{equation-ry} implies $f_0^*(i_*\cG)$ has a filtration with 
factors $j_{\ell*}(\cG \otimes \cL_{Z}^{-\ell})$ for $\ell = 0, 1, \dots, n-1$. Since 
$\cT$ contains $f_0^*(i_*\cG)$ and all of these factors for $\ell \neq k$, it follows that 
$\cT$ also contains $j_{k*}(\cG \otimes \cL_{Z}^{-k})$. But the twist by a power of 
$\cL_{Z}$ is an autoequivalence of $\Db(Z)$, so this proves the claim.

Now take any $\cF \in \Db(X)^{\bmu_n}$. To finish the proof we must show $\cF \in \cT$. 
The canonical morphism $f_0^*f_{0*}\cF \to \cF$ is an isomorphism on $X \setminus Z$ 
by~\cite[Corollary~5.3]{elagin2014equivariant}, since $X \setminus Z \to Y \setminus Z$ is 
an \'etale Galois cover. 
Hence the cone $\cF'$ of this morphism is supported set-theoretically on~$Z$. 
It follows that each cohomology of $\cF'$ is supported set-theoretically on~$Z$, 
and hence admits a filtration by sheaves supported on $Z$ scheme-theoretically. 
Moreover, this filtration can be chosen $\bmu_n$-equivariantly. 
By~\eqref{equation-dbz-equivariant} any $\bmu_n$-equivariant sheaf scheme-theoretically supported on~$Z$ 
can be written as a direct sum of sheaves contained in the categories $j_{\ell*}\Db(Z)$. 
Since $\cT$ contains all these categories, it follows that $\cF'$ and hence $\cF$ are contained 
in $\cT$.
\end{proof}


\section{Semiorthogonal decompositions induced by a Lefschetz decomposition}
\label{section-induced-sods}

Let $Y$ be an algebraic variety with a line bundle $\cO_Y(1)$, and assume
a rectangular Lefschetz decomposition
\begin{equation}
\label{equation-dby}
\Db(Y) = \langle \cB, \cB(1), \dots, \cB(m-1) \rangle
\end{equation}
is given (see Section~\ref{section-applications} for examples of such $Y$).
Here $\cB(t)$ denotes the image of $\cB$ under the ``twist by $t$'' 
autoequivalence 
\begin{equation*}
\cF \mapsto \cF(t) := \cF \otimes \cO_Y(t)
\end{equation*} 
of $\Db(Y)$. 
We denote by $\beta:\cB \to \Db(Y)$ the embedding functor. By~\eqref{equation-dby}
it has a left adjoint which we denote by $\beta^*$. Moreover, twisting~\eqref{equation-dby}
by $\cO_Y(-(m-1))$, we see $\beta$ also has a right adjoint, which we denote by $\beta^!$.

For any variety $X$ mapping to $Y$, we define $\cO_X(1)$ as the pullback of 
$\cO_Y(1)$ and use the same notation as above for twists. 
We show~\eqref{equation-dby} induces semiorthogonal decompositions
of the derived categories of a cyclic cover of $Y$ and a divisor in $Y$.

\begin{lemma}
\label{lemma-dbx}
Let $n$ and $d$ be positive integers such that $(n-1)d < m$. 
Let $f: X \to Y$ be a degree $n$ cyclic cover of $Y$ ramified over a divisor in $|\cO_Y(nd)|$. 
\begin{enumerate}
\item 
\label{lemma-dbx-ff}
The functor $f^*:\Db(Y) \to \Db(X)$ is fully faithful on the subcategory 
$\cB \subset \Db(Y)$. 
\item Denoting the essential image of $\cB$ by $\cB_X$, for each $0 \le t \le m - (n-1)d$ there is a semiorthogonal
decomposition
\begin{equation}
\label{equation-dbx}
\Db(X) = \langle \cB_X, \dots, \cB_X(t-1), \cA_X(t), \cB_X(t), \dots, \cB_X(m-(n-1)d-1) \rangle,
\end{equation}
where $\cA_X$ is the full subcategory of $\Db(X)$ consisting of all $\cG \in \Db(X)$ 
such that $f_*\cG \in \langle \cB(-(n-1)d), \dots, \cB(-1) \rangle$.
\end{enumerate}
\end{lemma}

\begin{remark}
The semiorthogonal decomposition~\eqref{equation-dbx} still holds for $(n-1)d = m$ --- 
in this case there are no ``trivial components'' equivalent to $\cB$ in $\Db(X)$,
so $\cA_X = \Db(X)$.
\end{remark}

\begin{proof}
Let $\cF, \cG \in \cB$. For any integers $r$ and $s$, adjunction gives
\begin{equation*}
\Hom(f^*\cF(r), f^*\cG(s)) = \Hom(\cF(r), f_*f^*\cG(s)).
\end{equation*}
By the projection formula and~\eqref{equation-ry} we have
\begin{equation*}
f_*f^*\cG = \cG \otimes (\cO_Y \oplus \cO_Y(-d) \oplus \cdots \oplus \cO_Y(-(n-1)d)).
\end{equation*}
Hence
\begin{equation}
\label{equation-adjunction}
\Hom(f^*\cF(r), f^*\cG(s)) = \bigoplus_{a = 0}^{n-1} \Hom(\cF(r), \cG(s-ad)).
\end{equation}

If $r = s$, the $a > 0$ summands vanish by semiorthogonality of~\eqref{equation-dby} 
and the assumption $(n-1)d < m$. 
This proves $f^*$ is fully faithful on $\cB(r)$.

If  $0 \leq s < r \leq m - (n-1)d - 1$, all of the summands in~\eqref{equation-adjunction} vanish, 
again by semiorthogonality of~\eqref{equation-dby}.
This proves the sequence
\begin{equation*}
 \cB_X, \cB_X(1), \dots, \cB_X(m-(n-1)d-1)
\end{equation*}
is semiorthogonal. 
Note that $\cB_X$ is admissible in $\Db(X)$. Indeed, since $\cB$ is admissible in~$\Db(Y)$, 
it suffices to observe $f^*$ has right and left adjoints; the existence of the 
right adjoint is obvious, while the left adjoint exists since the Grothendieck duality 
functor $f^!$ satisfies
\begin{equation}
\label{equation-f-shriek}
f^!(\cF) = f^*(\cF) \otimes \cO_X((n-1)d)
\end{equation}
and has a left adjoint. Hence for $0 \le t \le m - (n-1)d$ there is a semiorthogonal decomposition
\begin{equation*}
\Db(X) = \langle \cB_X, \dots, \cB_X(t-1), \cA_X^{(t)}, \cB_X(t), \dots, \cB_X(m-(n-1)d-1) \rangle,
\end{equation*}
where $\cA_X^{(t)} \subset \Db(X)$ is the full subcategory of $\cG \in \Db(X)$ such that for all $\cF \in \cB$ 
we have
\begin{align*}
\Hom(f^*\cF(s),\cG) = 0&\quad\text{for all $t \le s \le m-(n-1)d - 1$, }\\
\Hom(\cG,f^*\cF(s)) = 0&\quad\text{for all $0 \le s \le t - 1$.}
\end{align*}
To finish the proof, we must show $\cA_X^{(t)} = \cA_X(t)$. 
By adjunction and~\eqref{equation-f-shriek}, the above conditions can be rewritten as
\begin{align*}
\Hom(\cF(s),f_*\cG) = 0&\quad\text{for all $t \le s \le m-(n-1)d - 1$,}\\
\Hom(f_*\cG,\cF(s-(n-1)d)) = 0&\quad\text{for all $0 \le s \le t - 1$,}
\end{align*}
or equivalently
\begin{equation*}
f_*\cG \in {}^\perp\langle \cB(-(n-1)d),\dots,\cB(t-1-(n-1)d) \rangle \cap \langle \cB(t),
\dots, \cB(m-(n-1)d-1) \rangle^\perp.
\end{equation*}
It follows from~\eqref{equation-dby} twisted by $\cO_Y(-(n-1)d)$ that the  
above intersection of categories equals 
\begin{equation*}
\langle \cB(t-(n-1)d), \dots, \cB(t-1) \rangle.
\end{equation*}
This is the twist by $\cO_Y(t)$ of the category defining $\cA_X$, hence 
$\cA_X^{(t)} = \cA_X(t)$.
\end{proof}

Later we will also need the following strengthening of Lemma~\ref{lemma-dbx}\eqref{lemma-dbx-ff}.
Under the stronger assumption $nd \leq m$ (as in the setup of Theorem~\ref{main-theorem}), 
it says that the functor $f^*$ is fully faithful not only on the component $\cB$, 
but also on the subcategory generated by $d$ twists of~$\cB$.

\begin{lemma}
\label{faithful-on-bd}
Let $n$ and $d$ be positive integers such that $nd \leq m$. 
Let $f: X \to Y$ be a degree~$n$ cyclic cover of $Y$ ramified over a divisor in $|\cO_Y(nd)|$. 
Then the restriction of the functor $f^*: \Db(Y) \to \Db(X)$ to the subcategory
\begin{equation*}
\langle \cB, \cB(1), \dots, \cB(d-1) \rangle \subset \Db(Y) 
\end{equation*}
is fully faithful and induces an equivalence onto  
$\langle \cB_X, \cB_X(1), \dots, \cB_X(d-1) \rangle \subset \Db(X)$.
\end{lemma}

\begin{proof}
The same argument as in the proof of part~\eqref{lemma-dbx-ff} of Lemma~\ref{lemma-dbx} works.
\end{proof}

Note that the action of $\bmu_n$ on $\Db(X)$ preserves the decomposition~\eqref{equation-dbx} 
of Lemma~\ref{lemma-dbx}. 
Moreover, since the twist by $t$ autoequivalence of $\Db(X)$ is $\bmu_n$-equivariant, 
the category $\cB_X(t)^{\bmu_n}$ equals the category $\cB_X^{\bmu_n}(t)$ 
obtained from $\cB_X^{\bmu_n}$ by twisting by $\cO_X(t)$; similarly, 
$\cA_X(t)^{\bmu_n}$ equals $\cA_X^{\bmu_n}(t)$. Hence applying 
Theorem~\ref{theorem-group-action-sod} to~\eqref{equation-dbx} gives the following.

\begin{lemma}
\label{lemma-dbx-equivariant}
Let $n$ and $d$ be positive integers such that $(n-1)d < m$. 
Let ${f: X \to Y}$ be a degree $n$ cyclic cover of $Y$ ramified over a divisor in $|\cO_Y(nd)|$. 
For each ${0 \le t \le m - (n-1)d}$ there is a semiorthogonal decomposition
\begin{equation}
\label{equation-dbx-equivariant}
\Db(X)^{\bmu_n} = \langle 
\cB_X^{\bmu_n}, \dots, \cB_X^{\bmu_n}(t-1), \cA_X^{\bmu_n}(t), \cB_X^{\bmu_n}(t), \dots, \cB_X^{\bmu_n}(m - (n-1)d -1) \rangle.
\end{equation}
\end{lemma}

For a divisor in $Y$, we have the following analogue of Lemma~\ref{lemma-dbx}.
\begin{lemma}
\label{lemma-dbz}
Let $e$ be an integer such that $1 \leq e < m$. Let $i: Z \hookrightarrow Y$ be the
inclusion of a divisor in $|\cO_Y(e)|$.
\begin{enumerate}
\item 
\label{lemma-dbz-ff}
The functor $i^*:\Db(Y) \to \Db(Z)$ is fully faithful on the subcategory $\cB \subset \Db(Y)$.
\item Denoting the essential image of $\cB$ by $\cB_Z$, for each $0 \le t \le m - e$ there is a 
semiorthogonal decomposition
\begin{equation}
\label{equation-dbz}
\Db(Z) = \langle \cB_Z, \dots, \cB_Z(t-1), \cA_Z(t), \cB_Z(t), \dots, \cB_Z(m-e-1) \rangle,
\end{equation}
where $\cA_Z$ is the full subcategory of $\Db(Z)$ consisting of all $\cG \in \Db(Z)$ such that
$i_*\cG \in \langle \cB(-e), \dots, \cB(-1) \rangle$.
\end{enumerate}
\end{lemma}

\begin{remark}
Again, the semiorthogonal decomposition~\eqref{equation-dbz} still holds for $e = m$ --- 
in this case there are no ``trivial components'' equivalent to $\cB$ in $\Db(Z)$,
so $\cA_Z = \Db(Z)$.
\end{remark}

\begin{proof}
Let $\cF,\cG \in \cB$. For any integers $r$ and $s$, adjunction gives
\begin{equation*}
\Hom(i^*\cF(r),i^*\cG(s)) = \Hom(\cF(r),i_*i^*\cG(s)).
\end{equation*}
On the other hand, we have a distinguished triangle
\begin{equation*}
\cG(s-e) \to \cG(s) \to i_*i^*\cG(s)
\end{equation*}
obtained by tensoring the resolution of $i_*\cO_Z$ on $Y$ with $\cG(s)$. Applying 
$\Hom(\cF(r),-)$ gives a long exact sequence 
\begin{equation*}
\cdots \to \Hom(\cF(r), \cG(s-e)) \to \Hom(\cF(r), \cG(s)) \to \Hom(\cF(r), i_*i^*\cG(s)) \to \cdots
\end{equation*}
Now the result follows by the same argument as in the proof of Lemma~\ref{lemma-dbx}, 
using the above sequence in place of~\eqref{equation-adjunction}.
\end{proof}

\begin{remark}
Lemmas~\ref{lemma-dbx} and~\ref{lemma-dbz} generalize directly to the situation where 
the Lefschetz decomposition~\eqref{equation-dby} is not assumed to be rectangular. 
However, we will not need this generalization.
\end{remark}


\section{Proof of the main result}
\label{section-proof}

In this section, we prove Theorem~\ref{main-theorem}. The functors 
embedding the $n-1$ copies of the category $\cA_Z$ into $\cA_X^{\bmu_n}$ are 
constructed explicitly in the course of the proof.

\subsection{Setup}
\label{subsection-theorem-setup}
Recall the setup of the theorem: 
Let $Y$ be an algebraic variety with a line bundle $\cO_Y(1)$ 
and a length $m$ rectangular Lefschetz decomposition of its derived 
category as in~\eqref{equation-dby}. 
Choose positive integers $n$ and $d$ such that $nd \leq m$. 
We set $\cL = \cO_Y(d)$ and consider a degree~$n$ cyclic cover $f: X \to Y$ 
ramified over a Cartier divisor $Z$ in $|\cL^n| = |\cO_Y(nd)|$, as in Section~\ref{section-cyclic-cover}. 
This data fits into a commutative diagram~\eqref{diagram-xyz}.

\subsection{Strategy of the proof}
We start with one of the semiorthogonal
decompositions of $\Db(X)^{\bmu_n}$ provided by Theorem~\ref{theorem-cyclic-cover-sod}, 
namely
\begin{equation}
\label{new-proof-sod-1}
\Db(X)^{\bmu_n} = \langle f_0^*\Db(Y), j_{0*}\Db(Z), j_{1*}\Db(Z), \dots, j_{n-2*}\Db(Z) \rangle.
\end{equation}
Taking into account the decomposition of $\Db(Y)$ given by~\eqref{equation-dby} 
and of $\Db(Z)$ given by~\eqref{equation-dbz} with $t = 0$ and $e = nd$,
we see $\Db(X)^{\bmu_n}$ has a semiorthogonal decomposition with 
\begin{equation*}
m + (n-1)(m-nd) = nm - n(n-1)d 
\end{equation*}
copies of the category $\cB$ and $n-1$ copies of the category $\cA_Z$ as components. 
On the other hand, Lemma~\ref{lemma-dbx-equivariant} gives
\begin{equation}
\label{new-proof-sod-0}
\Db(X)^{\bmu_n} = \langle \cA_X^{\bmu_n}, 
\cB_X^{\bmu_n}, \cB_X^{\bmu_n}(1), \dots, \cB_X^{\bmu_n}(m - (n-1)d -1) \rangle.
\end{equation} 
Note that the action of $\bmu_n$ on $\cB_X(t)$ is trivial for any $t$, so by Proposition~\ref{proposition-trivial-action-sod} 
it follows
\begin{equation*}
\cB_X^{\bmu_{n}}(t) = \langle \cB_X(t) \otimes \one, \dots, \cB_X(t) \otimes \chi^{n-1} \rangle.
\end{equation*}
Hence the decomposition~\eqref{new-proof-sod-0} has 
${n(m-(n-1)d)}$ copies of $\cB$ (the same number as above!) and
one copy of $\cA_X^{\bmu_n}$ as components. 
To prove Theorem~\ref{main-theorem}, we find a sequence of mutations 
transforming the $\cB$-components of~\eqref{new-proof-sod-1} 
into those of~\eqref{new-proof-sod-0}.

To concisely write the decompositions occurring in the proof, we introduce the 
following notation. Given subsets of ``twists'' $T \subset \bZ$ and ``weights'' 
$W \subset \bZ/n$, we define
\begin{equation*}
\cB_{X}^W(T)  = \langle \cB_X(t)\otimes \chi^k \rangle_{t \in T, \, k \in W} \subset \Db(X)^{\bmu_n} 
\end{equation*}
to be the triangulated subcategory generated by $\cB_X(t)\otimes \chi^k$ for $t \in T, k \in W$, 
and we define
\begin{equation*}
\cB_{Z}(T)  = \langle \cB_Z(t) \rangle_{t \in T} \subset \Db(Z)
\end{equation*}
to be the triangulated subcategory generated by $\cB_Z(t)$ for $t \in T$.
If $a \leq b$ are integers, we write $[a,b]$ for the set of integers $t$ with 
$a \leq t \leq b$.
We also set
\begin{equation*}
M = m - (n-1)d.
\end{equation*}
With this notation, we can rewrite~\eqref{new-proof-sod-1} as
\begin{equation}
\label{new-proof-sod-2}
\Db(X)^{\bmu_n} = \langle \cB_X^0([0,m-1]), j_{0*}\Db(Z), j_{1*}\Db(Z), \dots, j_{n-2*}\Db(Z) \rangle
\end{equation}
and~\eqref{new-proof-sod-0} as
\begin{equation}
\label{new-proof-sod-final}
\Db(X)^{\bmu_n} = \langle \cA_X^{\bmu_n}, \cB_X^{[0,n-1]}([0,M-1]) \rangle.
\end{equation}

\subsection{Mutations}
Now we perform a sequence of mutations.

\medskip \noindent
\textbf{Step 1.} Write the first component of the decomposition~\eqref{new-proof-sod-2} as 
\begin{equation*}
\cB_X^0([0,m-1]) = \langle \cB_X^0([0,M-1]), \cB_X^0([M,M+d-1]),
\dots,\cB_X^0([m-d,m-1]) \rangle,
\end{equation*}
with $M$ copies of $\cB$ in the first component and $d$ copies in each of the next $n-1$ components.
Note that the subcategory $j_{k*}\Db(Z) \subset \Db(X)^{\bmu_n}$ is admissible since the functor
$j_{k*}$ has both left and right adjoints $j_k^*$ and $j_k^!$. Hence the mutation functors
through this subcategory are well-defined.
So for $a = 1, \dots, n-1$, we can successively right mutate the component $\cB_X^0([m-ad,m-ad+d-1])$ 
through
\begin{equation*}
\langle j_{0*} \Db(Z), \dots, j_{n-a-1*}\Db(Z) \rangle
\end{equation*}
in~\eqref{new-proof-sod-2}. To understand the result we need the following lemma.

\begin{lemma}
\label{lemma-mutation-1}
For any twist $t \in \bZ$ and weight $k \in \bZ/n$, we have
\begin{equation*}
\rR_{j_{k*} \Db(Z)}(\cB_{X}^k(t)) = \cB_{X}^{k+1}(t-d).
\end{equation*}
\end{lemma}

\begin{proof}
In fact, for any $\cF \in \Db(Y)$ we prove
\begin{equation*}
\rR_{j_{k*} \Db(Z)}(f_k^*\cF) \cong f_{k+1}^*\cF(-d).
\end{equation*}
Indeed, tensoring the exact sequence~\eqref{equation-resolution-z} by $f_k^*\cF = f_0^*\cF \otimes \chi^k$ and using 
the projection formula $f_0^*\cF \otimes j_{0*}\cO_Z \cong j_{0*}i^*\cF$, 
we obtain a distinguished triangle
\begin{equation}
\label{equation-mutation-triangle}
f_{k+1}^*\cF(-d) \to f_k^*\cF \to j_{k*}i^*\cF.
\end{equation}
The last vertex is in $j_{k*}\Db(Z)$ and the first is in $f_{k+1}^*\Db(Y)$, so to show this 
is a mutation triangle~\eqref{right-mutation-triangle} it suffices to show the pair $(j_{k*}\Db(Z), f_{k+1}^*\Db(Y))$ is 
semiorthogonal. But this holds by~\eqref{sojf}.
\end{proof}

By an iterated application of Lemma~\ref{lemma-mutation-1}, the result of the above mutations is 
a semiorthogonal decomposition
\begin{multline}
\label{new-proof-sod-3}
\Db(X)^{\bmu_n} = \langle \cB_X^0([0,M-1]), \\
\qquad \quad j_{0*}\Db(Z), \cB_X^1([M-d,M-1]), \, \, \,
j_{1*}\Db(Z), \cB_X^2([M-d,M-1]),  \ \dots, \\
j_{n-2*}\Db(Z), \cB_X^{n-1}([M-d,M-1])  \rangle.
\end{multline}

\medskip \noindent
\textbf{Step 2.} 
Consider the twist by $\cO_Z(d)$ of the decomposition~\eqref{equation-dbz} for $t=0$ (note that in our situation $e = nd$)
\begin{equation*}
\Db(Z) = \langle \cA_Z(d), \cB_Z(d), \cB_Z(d+1) \dots, \cB_Z(m-nd+d-1) \rangle
= \langle \cA_Z(d), \cB_Z([d,M-1]) \rangle, 
\end{equation*}
and substitute it for each copy of $\Db(Z)$ in~\eqref{new-proof-sod-3}:
\begin{multline}
\label{new-proof-sod-4}
\Db(X)^{\bmu_n} = \langle \cB_X^0([0,M-1]), \\ 
j_{0*}\cA_Z(d), j_{0*}\cB_Z([d,M-1]), \cB_X^1([M-d,M-1]), \\
j_{1*}\cA_Z(d), j_{1*}\cB_Z([d,M-1]), \cB_X^2([M-d,M-1]), 
\hbox to 0pt{\ \dots,\hss} \\
j_{n-2*}\cA_Z(d), j_{n-2*}\cB_Z([d,M-1]), \cB_X^{n-1}([M-d,M-1]) \rangle.
\end{multline} 

\medskip \noindent
\textbf{Step 3.}
Note that the subcategory $\cB_X^k(t) \subset \Db(X)^{\bmu_n}$ is admissible for all $t$ and $k$, 
since $\cB(t)$ is admissible in $\Db(Y)$ and the functor $f_k^*$ has both right and left adjoints.
So for ${k = 0, \dots, n-2}$, we can successively left mutate the component $j_{k*}\cA_Z(d)$ 
through the copies of $\cB$ appearing to its left in~\eqref{new-proof-sod-4}.
This gives
\begin{equation}
\label{new-proof-sod-5}
\Db(X)^{\bmu_n} = \langle 
\Phi_0(\cA_Z),\Phi_1(\cA_Z), \dots, \Phi_{n-2}(\cA_Z), \cC_{n-1} \rangle,
\end{equation} 
where the functors $\Phi_k$ are defined by
$\Phi_k(\cF) = \rL_{\cC_k}(j_{k*}\cF(d))$ and
\begin{multline}
\label{equation-C-k}
\cC_k = \langle \cB_X^0([0,M-1]), \ j_{0*}\cB_Z([d,M-1]), \cB_X^1([M-d,M-1]), \\
\dots, 
j_{k-1*}\cB_Z([d,M-1]), \cB_X^k([M-d,M-1]) \rangle.
\end{multline}
Note that the functors $\Phi_k: \cA_Z \to \Db(X)^{\bmu_n}$ are fully faithful 
since $j_{k_*}\cA_Z(d)$ is left orthogonal to $\cC_k$.
We shall see their images lie in $\cA_X^{\bmu_n}$ and give 
the desired semiorthogonal decomposition.

\subsection{The final argument}
It remains to show the ``$\cB$-part'' $\cC_{n-1}$ 
of the decomposition~\eqref{new-proof-sod-5} equals the ``$\cB$-part'' 
$\cB_{X}^{[0,n-1]}([0, M-1])$ of the decomposition~\eqref{new-proof-sod-final}. 
We do this in Lemma~\ref{lemma-simplify-C-k}, 
where we in fact  establish a simple expression for each category $\cC_k$, 
which for $k = n-1$ gives the desired equality of ``$\cB$-parts.'' 
We will need the following mutation lemma.

\begin{lemma}
\label{lemma-mutation-2}
Assume $nd < m$. For twists $s,t \in \bZ$ and weights $k, \ell \in \bZ/n$, we have
\begin{equation}
\label{equation-mutation-2}
\rL_{\cB_X^{\ell}(s)}(j_{k*}\cB_Z(t)) = 
\begin{cases}
j_{k*}\cB_Z(t) & \text{if $k \neq \ell$ or if $t < s < t + M-d$,}\\
\cB_X^{k+1}(t-d) & \text{if $k = \ell$ and $t = s$.}
\end{cases}
\end{equation}
\end{lemma}

\begin{proof}
The assumption $nd < m$ guarantees $\cB_Z$ is defined. 
For $k \neq \ell$ or $t < s < t + M-d$, it suffices to show
the pair $(j_{k*}\cB_Z(t), \cB_X^{\ell}(s))$ is semiorthogonal. If $k \neq \ell$, 
this holds by~\eqref{sojf} since ${\cB_X^{\ell}(s) = f_{\ell}^*(\cB(s))}$. If $k = \ell$, 
note that by adjunction the desired semiorthogonality is equivalent to 
semiorthogonality of the pair $(\cB_Z(t), j_k^*\cB^k_X(s))$. Since 
$j_k^*\cB^k_X(s) = \cB_Z(s)$, this semiorthogonality 
holds if $t < s < t + M - d$ by Lemma~\ref{lemma-dbz}.

For $k = \ell$ and $t = s$, by definition of the category $\cB_Z(t)$ it is enough to check
\begin{equation}
\rL_{\cB_X^{k}(t)}(j_{k*}i^*\cF(t)) \cong f^*_{k+1}\cF(t-d)[1]
\end{equation}
for any $\cF \in \cB$. Twisting the triangle~\eqref{equation-mutation-triangle} by 
$\cO_X(t)$ and then rotating, we obtain a triangle
\begin{equation*}
 f_k^*\cF(t) \to j_{k*}i^*\cF(t) \to f_{k+1}^*\cF(t-d)[1].
\end{equation*}
The first vertex is in $\cB_X^k(t)$ and the last is in $\cB_X^{k+1}(t-d)$, so 
to show this is a mutation triangle it suffices to show the pair
$(\cB_X^{k+1}(t-d), \cB_X^k(t))$ is semiorthogonal. But this holds 
by the decomposition~\eqref{new-proof-sod-0} since $m - (n-1)d - 1 \ge d$.
\end{proof}

Now we can prove a simple formula for the categories $\cC_k$.

\begin{lemma}
\label{lemma-simplify-C-k}
For $0 \le k \le n-1$, there is an equality
\begin{equation}
\cC_k = \cB_X^{[0,k]}([0,M-1]).
\end{equation} 
\end{lemma}
\begin{proof}
If $m = nd$, the result is obvious. Indeed, in this case $M = d$ and there are no $\cB_Z$-components
in~\eqref{equation-C-k}.  Thus from now on we assume $nd < m$.

The proof is by induction on $k$. For $k = 0$, there is nothing to prove. If the result 
holds for~$k$, then
\begin{equation}
\label{equation-C-k+1}
\cC_{k+1} = \langle \cB_X^{[0,k]}([0,M-1]), j_{k*}\cB_Z([d,M-1]), \cB_X^{k+1}([M-d,M-1]) \rangle. 
\end{equation}
To show this equals $\cB_X^{[0,k+1]}([0,M-1])$, we mutate each component of 
$j_{k*}\cB_Z([d,M-1])$ to the left through a subset of the components of $\cB_X^{[0,k]}([0,M-1])$. 
Namely, for $t = d, \dots, M-1$, we successively left mutate $j_{k*}\cB_Z(t)$ through 
$\cB_X^{[0,k]}([t,M-1])$. 
By Lemma~\ref{lemma-mutation-2}, this transforms $j_{k*}\cB_Z(t)$ into 
$\cB_X^{k+1}(t-d)$. The components of the resulting decomposition of $\cC_{k+1}$ 
thus coincide (up to a permutation) with the components of $\cB_X^{[0,k+1]}([0,M-1])$. 
Hence these categories are equal.
\end{proof}

This finishes the proof of Theorem~\ref{main-theorem}. 

In the rest of the paper we frequently consider tensor product functors.
To unburden notation, we adopt the following convention. As soon as an object $E$
is given, the functor $\cF \mapsto \cF \otimes E$ will be denoted simply by $E$
(especially in formulas involving compositions of several functors as~\eqref{equation-phi-k} below). For instance, if $\cL$
is a line bundle on $X$, we will write simply $\cL$ for the twist by $\cL$ functor $\Db(X) \to \Db(X)$,
and if $\chi$ is a character of a group $G$, we will write simply $\chi$ for the twist by $\chi$
(considered as the trivial bundle with $G$-structure given by $\chi$) functor $\Db(X)^G \to \Db(X)^G$.

Using this convention, the result we have proved can be written as follows.
\begin{theorem}
\label{main-theorem-precise}
In the setup of subsection~\ref{subsection-theorem-setup},
the functors $\Phi_0,\Phi_1,\dots,\Phi_{n-2}:\cA_Z \to \cA_X^{\bmu_n}$ defined by
\begin{equation}
\label{equation-phi-k}
\Phi_k = \rL_{\cB_X^{[0,k]}([0,M-1])} \circ j_{k*} \circ \Tnsr{\cO_Z(d)}
\end{equation} 
are fully faithful, and their essential images give a semiorthogonal decomposition
\begin{equation*}
\cA_X^{\bmu_n} = \langle \Phi_0(\cA_Z), \Phi_1(\cA_Z), \dots, \Phi_{n-2}(\cA_Z) \rangle.
\end{equation*}
\end{theorem}

Below we give a simpler expression for the functors $\Phi_k$. 

\subsection{Simplifications of the functors $\Phi_k$} 
\label{subsection-simpler-Phi}

First we show the mutation functor in~\eqref{equation-phi-k} can be simplified considerably.

\begin{proposition}
\label{proposition-simpler-Phi}
For $0 \leq k \leq n-2$, 
there is an isomorphism of functors
\begin{equation}
\label{equation-simpler-Phi}
\Phi_k \cong \rL_{\cB_X^k([0,d-1])} \circ j_{k*} \circ \Tnsr{\cO_Z(d)}.
\end{equation}
\end{proposition}

\begin{proof}
The left mutation functor $\rL_{\cB_X^{[0,k]}([0,M-1])}$ factors into simpler pieces 
as follows.
By~\eqref{new-proof-sod-0} there is a decomposition
\begin{equation*}
\cB_X^{[0,k]}([0,M-1]) = \langle \cB_X^{[0,k]}([0,d-1]), \cB_X^{[0,k]}([d,M-1]) \rangle.
\end{equation*}
On the other hand, Lemma~\ref{faithful-on-bd} implies the action of $\bmu_n$ on $\cB_X([0,d-1])$ 
is trivial, so by the complete orthogonality in Proposition~\ref{proposition-trivial-action-sod} 
there is a decomposition
\begin{equation*}
\cB_X^{[0,k]}([0,M-1]) = \langle \cB_X^k([0,d-1]), \cB_X^{[0,k-1]}([0,d-1]), \cB_X^{[0,k]}([d,M-1]) \rangle.
\end{equation*}
Hence by~\eqref{left-mutation-composition} we get a factorization
\begin{equation*}
\rL_{\cB_X^{[0,k]}([0,M-1])} =
\rL_{\cB_X^{k}([0,d-1])} \circ \rL_{\cB_X^{[0,k-1]}([0,d-1])} \circ \rL_{\cB_X^{[0,k]}([d,M-1])}.
\end{equation*}
Thus, to prove the proposition it suffices to show the mutation functors 
$\rL_{\cB_X^{[0,k]}([d,M-1])}$ and $\rL_{\cB_X^{[0,k-1]}([0,d-1])}$
act as the identity functor on the category $j_{k*}\cA_Z(d)$.
This is a consequence of the following lemma.
\end{proof}

\begin{lemma}
\label{lemma-AZ-BX-orthogonal}
If $k \neq \ell$ or if $d \leq s \leq M-1$, the pair $(j_{k*}\cA_Z(d), \cB_X^{\ell}(s))$ is semiorthogonal. 
In particular, in this case $\rL_{\cB_X^{\ell}(s)}$ is the identity functor on $j_{k*}\cA_Z(d)$.
\end{lemma}

\begin{proof}
By adjunction, this is equivalent to semiorthogonality of the pair 
$(\cA_Z(d), j_k^* f_{\ell}^*(\cB(s)))$. If $k \neq \ell$, then $j_k^* \circ f_{\ell}^* = 0$, so 
this is clear. If $k = \ell$, then $j_k^* f_{\ell}^*(\cB(s)) = \cB_Z(s)$, and the required 
semiorthogonality follows from~\eqref{equation-dbz}.
\end{proof}

The proposition implies the functors $\Phi_k$ differ from each other by twists by 
characters:

\begin{corollary}
\label{corollary-Phi_k-Phi_0}
For $0 \leq k \leq n-2$, there is an isomorphism of functors
\begin{equation*}
\Phi_k \cong \Tnsr{\chi^k} \circ \Phi_0.
\end{equation*}
In particular, the semiorthogonal decomposition of Theorem~{\rm\ref{main-theorem-precise}} 
can be written as
\begin{equation*}
\cA_X^{\bmu_n} = \langle \Phi_0(\cA_Z), \Phi_0(\cA_Z) \otimes \chi, \dots, \Phi_{0}(\cA_Z) \otimes \chi^{n-2} \rangle.
\end{equation*}
\end{corollary}

\begin{proof}
We have
\begin{align*}
\Phi_k & \cong \rL_{\cB_X^k([0,d-1])} \circ j_{k*} \circ \Tnsr{\cO_Z(d)} \\
& \cong \rL_{\cB_X^0([0,d-1]) \otimes \chi^k} \circ \Tnsr{\chi^k} \circ j_{0*} \circ \Tnsr{\cO_Z(d)} \\
& \cong \Tnsr{\chi^k} \circ \rL_{\cB_X^0([0,d-1])} \circ j_{0*} \circ \Tnsr{\cO_Z(d)} \\
& \cong \Tnsr{\chi^k} \circ \Phi_0.
\end{align*}
The first and last isomorphisms hold by Proposition~\ref{proposition-simpler-Phi}, 
the second by the definitions, and the third by Lemma~\ref{lemma-mutation-autoequivalence}.
\end{proof}


\section{Rotation functors and reconstruction results}
\label{section-reconstruction}

\subsection{Reconstruction}

In~\cite{elagin2014equivariant} Elagin proved that, under certain conditions, 
an additive category equipped with a group action can be reconstructed from its 
equivariant category.

\begin{theorem}[\protect{\cite[Theorem~7.2]{elagin2014equivariant}}]
\label{theorem-elagin}
Let $G$ be a finite abelian group.
Let $\cC$ be an additive idempotent complete category, linear over an algebraically 
closed field of characteristic coprime to $|G|$.
Then there is an equivalence
\begin{equation*}
\cC \simeq (\cC^G)^{\widehat{G}},
\end{equation*}
where characters $\chi \in \widehat{G}$ act on $\cC^G$ by the tensor product functors $\Tnsr{\chi}:\cC^G \to \cC^G$.
\end{theorem}

\begin{remark}
Suppose in the situation of Theorem~\ref{theorem-elagin} that $\cC$ and $\cC^G$ are 
triangulated. Then $(\cC^G)^{\widehat{G}}$ comes with a natural class of distinguished triangles, 
consisting of those triangles whose image under the forgetful functor $(\cC^G)^{\widehat{G}} \to \cC^G$ 
are distinguished (see the discussion in Section~\ref{subsection-triangulated-equivariant-categories}). 
In fact, the equivalence $\cC \simeq (\cC^G)^{\widehat{G}}$ of the theorem respects the 
classes of distinguished triangles (in particular $(\cC^G)^{\widehat{G}}$ is triangulated). 
Indeed, unwinding Elagin's construction of this equivalence, it follows that its composition with the forgetful 
functor
\begin{equation*}
\cC \simeq (\cC^G)^{\widehat{G}} \to \cC^G
\end{equation*}
is the inflation functor sending an object $\cF$ to $\bigoplus_{g\in G} g^*\cF$ (with the obvious equivariant 
structure). As this composition is triangulated, it follows that $\cC \simeq (\cC^G)^{\widehat{G}}$ respects 
distinguished triangles.
\end{remark}

Theorem~\ref{theorem-elagin} applies to the category $\cA_X$ with the action of the group $\bmu_n$, 
where $X$ is as in Section~\ref{section-proof}. 
Note that in this case the dual group is $\hbmu_n  = \bZ/n$.

\begin{corollary}
\label{proposition-elagin}
In the setup of subsection~\ref{subsection-theorem-setup}, there is an equivalence 
$\cA_X \simeq (\cA_X^{\bmu_n})^{\bZ/n}$.
\end{corollary}

By Theorem~\ref{main-theorem-precise}, there is a semiorthogonal decomposition of 
$\cA_X^{\bmu_n}$ into $n-1$ copies of $\cA_Z$. In case $n = 2$, we have the following 
consequence of Theorem~\ref{theorem-elagin}.

\begin{corollary}
\label{main-corollary-precise}
Let $n = 2$ and let $\chi$ be the nontrivial character of $\bmu_2$. The functor 
\begin{equation}
\label{equation-tau}
\tau = \Phi_0^{-1} \circ \Tnsr{\chi} \circ \Phi_0: \cA_Z \to \cA_Z
\end{equation}
induces a $\bZ/2$-action on $\cA_Z$,
such that there is an equivalence
\begin{equation*}
\cA_X \simeq \cA_Z^{\bZ/2}.
\end{equation*}
\end{corollary}

The situation for $n > 2$ is more complicated. To recover $\cA_X$ from $\cA_Z$, 
we need the data of the gluing functors for the $n-1$ copies of $\cA_Z$ in the decomposition 
of Theorem~\ref{main-theorem-precise}, together with the action of $\hbmu_n$ on the gluing of 
these categories (see Section~\ref{subsection-reconstruction-n-geq-3} for more details).

In the rest of this section, we discuss some interesting autoequivalences of 
$\cA_X, \cA_X^{\bmu_n},$ and~$\cA_Z$, which we call \emph{rotation functors}. 
We use these rotation functors to identify more explicitly the functor $\tau$ 
from Corollary~\ref{main-corollary-precise} (see Proposition~\ref{proposition-tau}).
Then, in case $n=3$, we speculate about a way to reconstruct $\cA_X$ in terms of 
$\cA_Z$ and its associated rotation functor (see Section~\ref{subsection-reconstruction-n-geq-3}).

\subsection{Rotation functors}
\label{subsection-rotation-functors}
We work in the following setup: 
$Y$ is a variety with a rectangular Lefschetz decomposition as in~\eqref{equation-dby};
$f: X \to Y$ is a degree $n$ cyclic cover ramified over a divisor in $|\cO_Y(nd)|$, where 
$1 \leq (n-1)d \leq m$; and $i: Z \hookrightarrow Y$ is the inclusion of a divisor in $|\cO_Y(e)|$, 
where $1 \leq e \leq m$. 
This is the natural setup for defining the rotation functors. 
Later we specialize to the setup of subsection~\ref{subsection-theorem-setup}. 

The \emph{rotation functors} are the following 
compositions of twists with mutation functors:
\begin{align*}
\rL_{\cB_X} \circ \Tnsr{\cO_X(1)} & : \Db(X) \to \Db(X), \\
\rL_{\cB^{\bmu_n}_X} \circ \Tnsr{\cO_X(1)} &: \Db(X)^{\bmu_n} \to \Db(X)^{\bmu_n}, \\
\rL_{\cB_Z} \circ \Tnsr{\cO_Z(1)} &: \Db(Z) \to \Db(Z).
\end{align*}

If $(n-1)d = m$ the category $\cB_X$ is not defined, and if $e=m$
the category $\cB_Z$ is not defined. 
However, in these cases there are still natural definitions of the functors 
$\rL_{\cB_X}, \rL_{\cB_X^{\bmu_n}}, \rL_{\cB_Z}$, under an additional technical assumption --- 
the \emph{finiteness of the cohomological amplitude} (see~\mbox{\cite[Section~2.6]{kuznetsov-base-change}}) 
of the projection functor $\beta\beta^!$ onto $\cB$ (which holds automatically if $Y$ is smooth).
We discuss a definition of $\rL_{\cB_Z}$ in this case, the other functors being similar. We take
\begin{equation}\label{equation-spherical-twist}
\rL_{\cB_Z} = \Cone(i^*\beta\beta^!i_* \to \id).
\end{equation}
To make sense of this as a functor, we note that under the above assumption of finiteness of cohomological 
amplitude, the projection functor $\beta\beta^!$ can be represented as a Fourier--Mukai functor 
by~\cite[Theorem~7.1]{kuznetsov-base-change}. It follows that $i^*\beta\beta^!i_*$ is a Fourier--Mukai functor as well. 
Moreover, the morphism $i^*\beta\beta^!i_* \to \id$ is induced by a morphism of kernels of Fourier--Mukai functors. 
We define $\rL_{\cB_Z}$ to be the Fourier--Mukai functor with kernel given by the cone of this morphism 
of kernels. 

For the most part the reader can ignore the distinction between the $e = m$ functor $\rL_{\cB_Z}$ 
and the usual mutation functors, as they satisfy similar properties, 
e.g.\ for $\cF \in \Db(Z)$ there is a functorial distinguished triangle
\begin{equation*}
i^*\beta\beta^!i_*\cF \to \cF \to \rL_{\cB_Z}(\cF),
\end{equation*}
and the obvious analogue of Lemma~\ref{lemma-mutation-autoequivalence} 
holds. 

\begin{remark}
When $e = m$,
the functor $i^*\circ\beta:\cB \to \Db(Z)$ is in fact \emph{spherical} 
and $\rL_{\cB_Z}$ is the corresponding \emph{spherical twist}, but we will not need this fact.
\end{remark}

In what follows, when considering $\rL_{\cB_X}, \rL_{\cB_X^{\bmu_n}},$ or $\rL_{\cB_Z}$ in 
the boundary cases $(n-1)d = m$ or $e = m$, 
we will tacitly assume the projection functor 
$\beta\beta^!$ has finite cohomological amplitude. 
Again, this condition is automatic if $Y$ is smooth.

\begin{lemma}
The rotation functors preserve the subcategories $\cA_X$, $\cA_X^{\bmu_n}$, and $\cA_Z$.
\end{lemma}

\begin{proof}
We give the proof for $\cA_Z$, the other two cases being essentially the same. 
If $e = m$, then $\cA_Z = \Db(Z)$ and there is nothing to prove.
Thus we assume $e < m$, so that $\cB_Z$ is defined and $\rL_{\cB_Z}$ is the mutation functor.
Consider the semiorthogonal 
decomposition~\eqref{equation-dbz} for $t = 1$:
\begin{equation*}
\Db(Z) = \langle \cB_Z, \cA_Z(1), \cB_Z(1), \dots, \cB_Z(m-e-1) \rangle.
\end{equation*}
By Proposition~\ref{prop-mutations}, the functor $\rL_{\cB_Z}$ is fully faithful on $\cA_Z(1)$
and induces a semiorthogonal decomposition 
\begin{equation*}
\Db(Z) = \langle \rL_{\cB_Z}(\cA_Z(1)), \cB_Z, \cB_Z(1), \dots, \cB_Z(m-e-1) \rangle.
\end{equation*}
Comparing this with~\eqref{equation-dbz} for $t = 0$, we deduce the claim.
\end{proof}

The lemma shows the rotation functors restrict to endofunctors of 
$\cA_X$, $\cA_X^{\bmu_n}$, and $\cA_Z$, which we denote by 
\begin{align*}
\sO_X & : \cA_X \to \cA_X, \\
\sO_X^{\bmu_n} & : \cA_X^{\bmu_n} \to \cA_X^{\bmu_n}, \\
\sO_Z & : \cA_Z \to \cA_Z.
\end{align*}
In fact, the proof of the lemma shows the first two functors are 
autoequivalences provided $(n-1)d < m$, and the third is an autoequivalence 
provided $e < m$, with inverses given by the composition of a right mutation and a twist. 
(The first two are also autoequivalences if $n=2$ and $d=m$, and the third is if $e = m$, as in these cases 
$\rL_{\cB_X}, \rL_{\cB_X^{\bmu_n}},$ and~$\rL_{\cB_Z}$ 
are spherical twists.)

The following result is given by \cite[Corollary~3.18]{kuznetsov2015calabi} 
specialized to our present setup, see~\cite[Examples 3.1--3.3]{kuznetsov2015calabi}.

\begin{theorem}
\label{theorem-power-O}
If $f: X \to Y$ is a double cover ramified over a divisor in $|\cO_Y(2d)|$, where 
$1 \leq d \leq m$ (i.e. $n=2$ in the above setup), then the associated rotation functors satisfy 
\begin{align}
\sO_X^d & \cong \sigma [1], \\
(\sO_X^{\bmu_2})^d & \cong \Tnsr{\chi}[1], 
\end{align} 
where $\sigma: \cA_X \to \cA_X$ is the involution induced by the 
involution of $X$ over $Y$.

If $i: Z \hookrightarrow Y$ is the inclusion of a divisor in $|\cO_Y(e)|$, where $1 \leq e \leq m$, 
then the associated rotation functor satisfies 
\begin{equation}
\label{power-O_Z}
\sO_Z^e \cong [2].
\end{equation}
\end{theorem}

\begin{remark}
\label{remark-fractional-CY}
If $Y$, $X$, and $Z$ are smooth and projective, $n = 2$, and $\omega_Y = \cO_Y(-m)$, then 
Theorem~\ref{theorem-power-O} can be used to give an expression for a power of the Serre 
functors of $\cA_X, \cA_X^{\bmu_n}$, and $\cA_Z$ 
(see~\cite[Corollaries~3.7--3.9]{kuznetsov2015calabi}). 
In fact, this was the first author's original motivation for proving the theorem. 
\end{remark}

\subsection{Intertwining property}
For the rest of this section, we assume the setup and notation of Section~\ref{section-proof}.
In particular, from now on $nd \leq m$ and $Z$ is the branch divisor of $f: X \to Y$.

Here is a key observation about the rotation functors:

\begin{proposition}
\label{proposition-intertwining}
The functor $\Phi_0:\cA_Z \to \cA_X^{\bmu_n}$ defined by~\eqref{equation-phi-k} intertwines the rotation functors 
$\sO_Z$ and $\sO_X^{\bmu_n}$, i.e.
\begin{equation}
\label{equation-intertwining}
\sO_X^{\bmu_n} \circ \Phi_0 \cong \Phi_0 \circ \sO_Z.
\end{equation} 
\end{proposition}

\begin{proof}
We start by rewriting both sides of~\eqref{equation-intertwining}. For the left side, we have
\begin{align*}
\sO_X^{\bmu_n} \circ \Phi_0 & \cong 
\rL_{\cB_X^{\bmu_n}} \circ \Tnsr{\cO_X(1)} \circ \rL_{\cB_X^0([0,d-1])} \circ j_{0*} \circ \Tnsr{\cO_Z(d)} \\
& \cong \rL_{\cB_X^{\bmu_n}} \circ \rL_{\cB_X^0([1,d])} \circ j_{0*} \circ \Tnsr{\cO_Z(d+1)}.
\end{align*}
The first isomorphism holds by the definition of $\sO_X^{\bmu_n}$ and 
Proposition~\ref{proposition-simpler-Phi}, and the second by 
Lemma~\ref{lemma-mutation-autoequivalence} and the projection formula.
Recall that $\cB_X^{\bmu_n} = \cB_X^{[0,n-1]}$ and
note that $(\cB_X^{\bmu_n}, \cB_X^0([1,d-1]))$ is a semiorthogonal pair 
by~\eqref{new-proof-sod-0} and our assumption $nd \leq m$. 
(We caution the reader that the pair $(\cB_X^{\bmu_n}, \cB_X^0([1,d]))$ is 
not semiorthogonal if $m = nd$.)
The action of $\bmu_n$ on $\cB_X([0,d-1])$ is trivial by Lemma~\ref{faithful-on-bd}, 
so by the complete orthogonality in Proposition~\ref{proposition-trivial-action-sod} 
there is a decomposition 
\begin{equation*}
\langle \cB_X^{\bmu_n}, \cB_X^0([1,d-1]) \rangle = \langle \cB_X^0([0,d-1]), \cB_X^{[1,n-1]} \rangle.
\end{equation*}
Hence by~\eqref{left-mutation-composition} we have
\begin{equation*}
\rL_{\cB_X^{\bmu_n}} \circ \rL_{\cB_X^0([1,d])} \cong 
\rL_{\langle \cB_X^{\bmu_n}, \cB_X^0([1,d-1]) \rangle} \circ \rL_{\cB_X^0(d)} \\
\cong \rL_{\cB_X^0([0,d-1])} \circ \rL_{\cB_X^{[1,n-1]}} \circ \rL_{\cB_X^0(d)}.
\end{equation*}
Combining this with the above, we have
\begin{equation}
\label{intertwining-LHS}
\sO_X^{\bmu_n} \circ \Phi_0 \cong 
\rL_{\cB_X^0([0,d-1])} \circ \rL_{\cB_X^{[1,n-1]}} \circ \rL_{\cB_X^0(d)} 
\circ  j_{0*} \circ \Tnsr{\cO_Z(d+1)}.
\end{equation}

Now we consider the right side of~\eqref{equation-intertwining}. First we note 
\begin{equation*}
\Phi_0 \cong \rL_{\cB_X^0([0,d-1])} \circ  \rL_{\cB_X^{[1,n-1]}} \circ j_{0*} \circ \Tnsr{\cO_Z(d)}.
\end{equation*}
Indeed, $\rL_{\cB_X^{[1,n-1]}}$ is the identity on $j_{0*}\cA_Z(d)$ by 
Lemma~\ref{lemma-AZ-BX-orthogonal}, so this is equivalent to the isomorphism of
Proposition~\ref{proposition-simpler-Phi}. Using this and Lemma~\ref{lemma-mutation-autoequivalence}, 
we find
\begin{equation}
\label{intertwining-RHS}
\Phi_0 \circ \sO_Z \cong \rL_{\cB_X^0([0,d-1])} \circ  \rL_{\cB_X^{[1,n-1]}} \circ j_{0*} \circ \rL_{\cB_Z(d)} \circ \Tnsr{\cO_Z(d+1)}.
\end{equation}

To prove the proposition, by~\eqref{intertwining-LHS} and~\eqref{intertwining-RHS} it suffices 
to construct a morphism of functors 
\begin{equation*}
\rL_{\cB_X^0(d)} \circ  j_{0*} \to j_{0*} \circ \rL_{\cB_Z(d)}
\end{equation*}
whose composition with $\rL_{\cB_X^{[1,n-1]}}$ is an isomorphism. 
By Lemma~\ref{lemma-mutation-autoequivalence}, this is equivalent to constructing a morphism
\begin{equation*}
\rL_{\cB_X^0} \circ  j_{0*} \to j_{0*} \circ \rL_{\cB_Z}
\end{equation*}
whose composition with $\rL_{\cB_X^{[1,n-1]}(-d)}$ is an isomorphism. 

For this, consider the commutative diagram of functors 
\begin{equation}
\label{diagram-intertwining}
\vcenter{\xymatrix{
f_0^*\beta\beta^!f_{0*}j_{0*}  \ar[r] \ar[d] & j_{0*}  \ar[r] \ar@{=}[d] & \rL_{\cB_X^0} \circ j_{0*} \ar@{-->}[d] \\
j_{0*}i^*\beta\beta^!i_* \ar[r] & j_{0*} \ar[r] & j_{0*} \circ \rL_{\cB_Z}
}}
\end{equation}
Here the two rows come from the definition of the mutation functors 
(or from~\eqref{equation-spherical-twist} in case $m = nd$), 
and the left vertical arrow is induced by the isomorphism $f_{0*}j_{0*} \cong i_*$
and the morphism $f_0^* \to j_{0*}j_0^*f_0^* \cong j_{0*}i^*$ obtained from the unit of the adjunction 
between $j_0^*$ and~$j_{0*}$. It is easy to check the left square commutes. 
All of the functors in the diagram are Fourier--Mukai functors and the arrows in the 
diagram come from morphisms of kernels. In the corresponding diagram of 
Fourier--Mukai kernels we can find a dashed arrow making the diagram commute, 
and this induces the dashed arrow in the above diagram.

Now we describe the cone of the morphism $\rL_{\cB_X^0} \circ  j_{0*} \to j_{0*} \circ \rL_{\cB_Z}$ 
applied to an object $\cF \in \Db(Z)$. Set $\cG = \beta\beta^!i_*\cF$ so that $\cG \in \cB$.
Tensoring~\eqref{equation-resolution-z} with $f_0^*\cG$, we see 
the left column of diagram~\eqref{diagram-intertwining} applied to $\cF$ 
fits into a distinguished triangle
\begin{equation*}
f_0^*\cG(-d)\otimes\chi \to f_0^*\cG \to j_{0*}i^*\cG.
\end{equation*}
By the octahedral axiom, diagram~\eqref{diagram-intertwining} applied to $\cF$ thus 
gives a distinguished triangle
\begin{equation*}
f_0^*\cG(-d)\otimes\chi[1] \to \rL_{\cB_X^0}(j_{0*}\cF) \to j_{0*}\rL_{\cB_Z}(\cF).
\end{equation*}
The first vertex is contained in $\cB_X^1(-d)$, hence is killed by $\rL_{\cB_X^{[1,n-1]}(-d)}$. 
This proves the composition of $\rL_{\cB_X^0} \circ  j_{0*} \to j_{0*} \circ \rL_{\cB_Z}$ with 
$\rL_{\cB_X^{[1,n-1]}(-d)}$ is an isomorphism, as required.
\end{proof}

\subsection{The involution for $n = 2$}

As we observed in Corollary~\ref{main-corollary-precise}, if $n = 2$ there is an 
involution $\tau: \cA_Z \to \cA_Z$ such that $\cA_X \simeq \cA_Z^{\bZ/2}$, where 
$\bZ/2$ acts by $\tau$. Now we can give a simple formula for $\tau$ in terms of the 
rotation functor for $\cA_Z$.

\begin{proposition}
\label{proposition-tau}
If $n=2$ then 
\begin{equation}
\tau \cong \sO_Z^d[-1]
\end{equation}
is an involution of $\cA_Z$ such that $\cA_X \simeq \cA_Z^{\bZ/2}$, where 
$\bZ/2$ acts by $\tau$.
\end{proposition}

\begin{remark}
The proposition is consistent with the isomorphism $\sO_Z^{2d} \cong [2]$ given by Theorem~\ref{theorem-power-O}.
\end{remark}

\begin{proof}
This follows by combining Proposition~\ref{proposition-intertwining} with Theorem~\ref{theorem-power-O}.
Indeed, applying $d$ times the intertwining property~\eqref{equation-intertwining} we get an isomorphism
\begin{equation*}
(\sO_X^{\bmu_2})^d \circ \Phi_0 \cong \Phi_0 \circ \sO_Z^d.
\end{equation*}
Now $(\sO_X^{\bmu_2})^d \cong \Tnsr{\chi}[1]$ by Theorem~\ref{theorem-power-O}, so 
we have
\begin{equation*}
\Tnsr{\chi} \circ \Phi_0 \cong \Phi_0 \circ \sO_Z^d[-1].
\end{equation*}
Since $\tau = \Phi_0^{-1}\circ \Tnsr{\chi} \circ \Phi_0$ by~\eqref{equation-tau}, the result follows.
\end{proof}

\subsection{Reconstruction for $n>2$}
\label{subsection-reconstruction-n-geq-3}
As we already mentioned, the reconstruction of $\cA_X$ from $\cA_Z$ is 
more involved when $n > 2$.
First, Theorem~\ref{main-theorem-precise} gives a semiorthogonal decomposition
\begin{equation*}
\cA_X^{\bmu_n} = \langle \Phi_0(\cA_Z), \Phi_1(\cA_Z), \dots, \Phi_{n-2}(\cA_Z) \rangle.
\end{equation*}
In~\cite[Section~4]{kuznetsov-lunts} it is explained that given a semiorthogonal decomposition 
$\cT = \langle \cT_1, \cT_2 \rangle$, the category $\cT$ can be constructed as a ``gluing'' 
of the categories $\cT_1$ and $\cT_2$ along the ``gluing functor'' 
$i_1^! i_2[1]: \cT_2 \to \cT_1$, where $i_p: \cT_p \to \cT$ are the embeddings. 
(Technically, we should assume $\cT$ has a DG enhancement and work with the DG 
version of the gluing functor,
but in case $\cT_1$ and $\cT_2$ are admissible components of the derived category of a variety
and the gluing functor is of Fourier--Mukai type as in our situation, such a DG enhancement automatically exists).
In our case, it follows that $\cA_X^{\bmu_n}$ can be constructed as a gluing of $n-1$ 
copies of $\cA_Z$. In fact, one can check that for each adjacent pair of components in the 
above decomposition, the gluing functor $\cA_Z \to \cA_Z$ is given by 
the $d$-th power $\sO_Z^d$ of the rotation functor. 
The category glued from the $n-1$ copies of $\cA_Z$ carries an action of $\hbmu_n$ 
(induced by the action on $\cA_X^{\bmu_n}$), and it follows from Corollary~\ref{proposition-elagin} 
that $\cA_X$ is equivalent to the equivariant category for this action. Hence, in principle, 
$\cA_X$ can be entirely reconstructed from the category $\cA_Z$.

However, it is difficult to make this result explicit, because it is difficult to identify explicitly 
the action of $\hbmu_n$ on the glued category. Ideally, we would have a direct description of 
$\cA_X$ in terms of $\cA_Z$ and the rotation functor $\sO_Z$ (as we have when $n = 2$). 
In case $n = 3$, we propose to consider the category of distinguished triangles in $\cA_Z$ of the form
\begin{equation}
\label{triangle-category}
\cF_0 \to \sO_Z^{-d}[1](\cF_1) \to \sO_Z^{-2d}[2](\cF_2) \to \sO_Z^{-3d}[3](\cF_0).
\end{equation}
Note that $\sO_Z^{-3d} \cong [-2]$ by Theorem~\ref{theorem-power-O}, so 
indeed $\sO_Z^{-3d}[3](\cF_0) \cong \cF_0[1]$. Moreover, there is an action of 
$\bZ/3$ on the above category of triangles, where the generator acts by 
sending~\eqref{triangle-category} to the triangle 
\begin{equation*}
\cF_1 \to \sO_Z^{-d}[1](\cF_2) \to \sO_Z^{-2d}[2](\cF_0) \to \sO_Z^{-3d}[3](\cF_1)
\end{equation*}
obtained by applying $\sO_Z^{d}[-1]$ and rotating.
We think a natural guess is that the category $\cA_X$ is equivalent to the 
$\bZ/3$-equivariant category of the above category of triangles. 
Note that, a priori, it is not even clear the category of triangles is triangulated.


\section{Applications}
\label{section-applications}

The main results of this paper can be applied to a cyclic cover of any variety 
having a rectangular Lefschetz decomposition of its derived category. 
Quite a number of such varieties are known --- 
projective spaces (more generally projective bundles), 
Grassmannians $\G(k,n)$ for $k$ and $n$ coprime \cite{fonarev2013minimal} (and some of their linear sections), 
and some others (see~\cite[Section~4.1]{kuznetsov2015calabi} and~\cite{kuznetsov2014semiorthogonal}). 
If we consider more generally Deligne-Mumford stacks 
(see Remark~\ref{remark-stacks}), there are other natural examples, 
e.g.\ weighted projective spaces (regarded as quotient stacks).
Here we discuss only several examples of cyclic covers of varieties in the above list ---
quartic double solids, special Gushel--Mukai varieties, and cyclic cubic hypersurfaces.

\subsection{Quartic double solids}
Let $Y = \bP^3$ and let $f: X \to Y$ be a quartic double solid, i.e. a double cover of $Y$ ramified 
over a quartic surface $Z \in |\cO_Y(4)|$.
We have the standard semiorthogonal decomposition
\begin{equation*}
\Db(Y) = \langle \cO_Y, \cO_Y(1), \cO_Y(2), \cO_Y(3) \rangle.
\end{equation*}
Hence we are in the situation of Theorem~\ref{main-theorem} with $\cB = \langle \cO_Y \rangle$, 
$m = 4$, and $n = d = 2$.
The semiorthogonal decompositions~\eqref{equation-dbx} and~\eqref{equation-dbz} for $t=0$ 
are in this case 
\begin{align}
\nonumber
\Db(X) = \langle \cA_X, \cO_X, \cO_X(1) \rangle
& \qquad\text{and}\qquad
 \cA_Z = \Db(Z). 
\intertext{We note the Serre functor of $\cA_X$ satisfies 
$\rS_{\cA_X} \cong \sigma [2]$, where $\sigma$ is the involution generating 
the $\bmu_2$-action on $\cA_X$ (see~\cite[Corollary~4.6]{kuznetsov2015calabi}).
Applying Theorem~\ref{main-theorem}
and Corollary~\ref{main-corollary}, we conclude}
\label{quartic-double-solid-equivalences}
\cA_X^{\bmu_2} \simeq \Db(Z)
& \qquad\text{and}\qquad
\cA_X \simeq \Db(Z)^{\bZ/2}.
\end{align}
Here, by Proposition~\ref{proposition-tau} the group $\bZ/2$ acts on $\Db(Z)$ by $\sO_Z^2[-1]$, 
where $\sO_Z$ is the composition of the twist by $\cO_Z(1)$ with the spherical twist with respect to $\cO_Z$. 

This is exactly analogous to the relationship between the derived categories of an 
Enriques surface $S$ and its associated K3 surface $T$. Namely, there are equivalences 
\begin{equation*}
\Db(S)^{\bmu_2} \simeq \Db(T)
\qquad \text{and} \qquad 
\Db(S) \simeq \Db(T)^{\bZ/2} ,
\end{equation*}
where $\bmu_2$ acts on $\Db(S)$ by the shift of the Serre functor $\rS_{\Db(S)}[-2]$, i.e. by tensoring with~$\omega_S$, and 
$\bZ/2$ acts on $\Db(T)$ by the covering involution. 
Thus~\eqref{quartic-double-solid-equivalences} can be thought of as saying $\cA_X$ is a 
``noncommutative Enriques surface'' obtained by taking the quotient of the K3 surface~$Z$ 
by an involutive autoequivalence (which can be thought of as a ``noncommutative automorphism'').
This complements the results of~\cite{ingalls2010nodal}, where it is shown that if $X^{+}$ is 
a small resolution of singularities of an Artin--Mumford quartic double solid, then the 
distinguished component $\cA_X^{+} \subset \Db(X^{+})$ is equivalent to the distinguished component 
$\cA_S \subset \Db(S)$ of an associated Enriques surface $S$.

\subsection{Gushel--Mukai varieties}

Next we apply our results to the following class of varieties.
\begin{definition}
\label{definition-GM}
A smooth projective variety $X$ of dimension $N \geq 2$ is a 
\emph{Gushel--Mukai (GM) variety} if either: 
\begin{itemize}
\item $X$ is Fano with
\begin{equation*}  
\Pic(X) \cong \bZ,
\qquad
-K_X = (N - 2)H,
\qquad\text{and}\qquad
H^{N} = 10,
\end{equation*}
where $H$ is the ample generator of $\Pic(X)$; or
\item $X$ is a Brill--Noether general polarized K3 surface of degree $10$.
\end{itemize}
\end{definition}

We do not recall here the definition of a Brill--Noether general polarized K3, 
as below we focus on the Fano case. 
We note there is a more general definition of GM varieties, which includes 
singular varieties and curves, and coincides with the above definition 
for smooth varieties of dimension $\geq 2$. 
See~\cite{debarre-kuznetsov} for this and a detailed discussion of the geometry of GM varieties.

Let $V$ be a $5$-dimensional vector space. 
Let $\G(2,V)$ be the Grassmannian of $2$-dimensional subspaces of~$V$, 
embedded in $\bP(\bigwedge^2V) = \bP^9$ via the Pl\"{u}cker embedding.
The following theorem gives the classification of smooth GM varieties of dimension $\geq 2$, 
originally due to Gushel~\cite{gushel1983fano} and Mukai~\cite{mukai}.

\begin{theorem}[\cite{gushel1983fano, mukai, debarre-kuznetsov}]
\label{theorem-mukai}
Let $X$ be a smooth GM variety of dimension $N \geq 2$. 
Then there is a morphism 
\begin{equation*}
f: X \to \G(2,V)
\end{equation*}
such that precisely one of the following hold:
\begin{enumerate}
\item[(a)]  We have $2 \le N \le 5$ and there is a linear subspace $P \cong \bP^{N+4} \subset \bP(\bigwedge^2V)$ 
and a quadric hypersurface $Q \subset \bP(\bigwedge^2V)$ such that $f: X \to \G(2,V)$ embeds $X$ 
as a smooth complete intersection $X = \G(2,V) \cap P \cap Q$.

\item[(b)] 
We have $2 \le N \le 6$ and there is a linear subspace $P \cong \bP^{N+3} \subset \bP(\bigwedge^2V)$ 
and a quadric hypersurface $Q \subset \bP(\bigwedge^2V)$ such that $Y = \G(2,V) \cap P$ and 
$Z = Y \cap Q$ are smooth complete intersections, the image of $f$ is $Y$, and $f: X \to Y$ 
is the double cover of $Y$ ramified over $Z$.
\end{enumerate}
Conversely, any variety as in \textnormal{(a)} or \textnormal{(b)} is an $N$-dimensional GM variety. 
\end{theorem}

Let $X$ be a smooth GM variety of dimension $N \geq 2$. 
We call the morphism $f: X \to \G(2,V)$ of Theorem~\ref{theorem-mukai} 
the \emph{Gushel map} of $X$. 
We say $X$ is \emph{ordinary} if Theorem \ref{theorem-mukai}(a) holds, 
and \emph{special} if Theorem \ref{theorem-mukai}(b) holds. 
If $X$ is special, we often regard the Gushel map as a morphism $f: X \to Y$, where $Y$ is 
as in Theorem \ref{theorem-mukai}(b). 

From now on, assume $X$ is special and $N \geq 3$.
It follows from Theorem \ref{theorem-mukai} that 
the target $Y$ of the Gushel map $f:X \to Y$ is an $N$-dimensional linear section of the 
Grassmannian $\G(2,V)$, and the branch locus $Z \subset Y$ is an ordinary GM $(N-1)$-fold.
By~\cite[Section 6.1]{kuznetsov2006hyperplane}, we have a semiorthogonal decomposition
\begin{equation*}
\Db(Y) = \langle \cO_Y,\cU_Y^*,\dots,\cO_Y(N-2),\cU_Y^*(N-2) \rangle,
\end{equation*}
where $\cU_Y$ is the restriction to $Y$ of the tautological rank 2 bundle on $\G(2,V)$. 
We set ${\cB = \langle \cO_Y, \cU_Y^* \rangle}$, so that $\Db(Y)$ has a rectangular 
Lefschetz decomposition 
\begin{equation*}
\Db(Y) = \langle \cB, \cB(1), \dots, \cB(N-2) \rangle
\end{equation*}
of length $m = N-1$.
Since $X$ is a double cover of $Y$ ramified over a quadratic divisor, we are in the situation of
Theorem~\ref{main-theorem} with $n = 2$ and $d = 1$. Thus, we have 
decompositions
\begin{align}
\nonumber \Db(X) = \langle \cA_X, \cB_X, \dots, \cB_X(N-3) \rangle
& \qquad\text{and}\qquad
\Db(Z) = \langle \cA_Z, \cB_Z, \dots, \cB_Z(N-4) \rangle, 
\intertext{and equivalences}
\label{GM-equivalences} \cA_X^{\bmu_2} \simeq \cA_Z  
& \qquad\text{and}\qquad
\cA_X \simeq \cA_Z^{\bZ/2}.
\end{align}
Here, by Proposition~\ref{proposition-tau} the group $\bZ/2$ acts on $\cA_Z$ by $\sO_Z[-1]$.

An interesting feature of this example is that the categories $\cA_X$ and $\cA_Z$ are 
the distinguished components of GM varieties of dimensions differing by one. 
As is discussed in~\cite{X10-derived}, according to whether the dimension of a 
GM variety is even or odd, its distinguished component is a ``noncommutative K3 surface'' 
or a ``noncommutative Enriques surface'' 
(at the level of Serre functors this follows from \cite[Corollary~3.8]{kuznetsov2015calabi}).
Hence, the equivalences~\eqref{GM-equivalences} can be interpreted in the same way 
as~\eqref{quartic-double-solid-equivalences}, except now the K3 is also ``noncommutative'' 
in general.

\subsection{Cyclic hypersurfaces}
We say a hypersurface $X \subset \bP^N = \bP(V)$ of degree $n$ is \emph{cyclic} 
if it is invariant under the action of $\bmu_n$ induced by a representation 
of $\bmu_n$ on $V$ such that
\begin{equation*}
V \cong (W\otimes \one) \oplus \chi,
\end{equation*}
where $W \subset V$ is a subspace of codimension 1 and $\chi$ is the generator 
of $\hbmu_{n}$.
If we choose $\bmu_n$-equivariant 
coordinates $x_0,\dots,x_N$ on $V$ such that $W$ is given by the equation $x_0 = 0$,
then the equation of such a hypersurface has the form
\begin{equation*}
F(x_0,x_1,\dots,x_N) = x_0^n - G(x_1,\dots,x_N).
\end{equation*}
Consequently, $X$ is a cyclic covering of $Y = \bP(W)$ of degree $n$ ramified
over a hypersurface $Z \subset \bP(W)$ of degree~$n$ with equation $G = 0$.
Since $Y = \bP(W)$ admits a rectangular Lefschetz decomposition of its 
derived category
\begin{equation*}
\Db(Y) = \langle \cO_Y,\cO_Y(1),\dots,\cO_Y(N-1) \rangle,
\end{equation*}
we can apply our results with $\cB = \langle \cO_Y \rangle$, $m = N$, $d = 1$, 
and $n$ the degree of the hypersurface.
For $n \leq N$, we obtain semiorthogonal decompositions
\begin{align*}
\Db(X) &= \langle \cA_X, \cO_X,\dots,\cO_X(N-n) \rangle, \\
\Db(Z) &= \langle \cA_Z, \cO_Z,\dots,\cO_Z(N-n-1) \rangle,
\end{align*}
and a semiorthogonal decomposition of the equivariant category
\begin{equation*}
\cA_X^{\bmu_n} = \langle \cA_Z, \cA_Z\otimes\chi, \dots, \cA_Z\otimes\chi^{n-2} \rangle.
\end{equation*}
Here, we have used the form of Theorem~\ref{main-theorem} given by 
Corollary~\ref{corollary-Phi_k-Phi_0} and suppressed the embedding functor of 
$\cA_Z$ into $\cA_X^{\bmu_n}$.

Let us see what this gives for smooth cyclic cubic hypersurfaces of small dimension:
\begin{itemize}
\item 
If $X$ is a cyclic cubic surface, then $\cA_X = \langle \cO_X \rangle^\perp \subset \Db(X)$. 
From the description of $X$ as the blow-up of $\bP^2$ in $6$ points, it follows that
$\cA_X$ is generated by an exceptional collection of length $8$.
Further, $Z$ is an elliptic curve and $\cA_Z = \Db(Z)$. We get a decomposition
\begin{equation*}
\cA_X^{\bmu_3} = \langle \Db(Z), \Db(Z) \otimes \chi \rangle.
\end{equation*}
So, we have a category generated by an exceptional collection whose equivariant 
category decomposes into two copies of the derived category of an elliptic curve. 

\item 
If $X$ is a cyclic cubic threefold, then $\cA_X = \langle \cO_X, \cO_X(1) \rangle^\perp \subset \Db(X)$. 
In particular, $\cA_X$ is a fractional Calabi--Yau category of dimension $5/3$ 
(see~\cite[Corollary 4.3]{kuznetsov2004derived} or argue as in Remark~\ref{remark-fractional-CY}).
Further, $Z$ is a cubic surface and $\cA_Z = \langle \cO_Z \rangle^\perp \subset \Db(Z)$. 
In particular, $\cA_Z$ is generated by an exceptional collection of length $8$. 
We get a decomposition
\begin{equation*}
\cA_X^{\bmu_3} = \langle \cA_Z, \cA_Z \otimes \chi \rangle.
\end{equation*}
So, we have a fractional Calabi--Yau category whose equivariant category 
is generated by an exceptional collection of length 16. 
On the other hand,
applying the reconstruction result of Corollary~\ref{proposition-elagin}, we see
the fractional Calabi--Yau category $\cA_X$ can be obtained as the equivariant
category of a category generated by an exceptional collection of length 16.

\item 
If $X$ is a cyclic cubic fourfold, then $\cA_X = \langle \cO_X, \cO_X(1), \cO_X(2) \rangle^\perp \subset \Db(X)$. 
In particular, $\cA_X$ is a K3 category 
(again by~\cite[Corollary 4.3]{kuznetsov2004derived} or Remark~\ref{remark-fractional-CY}).
Further, $Z$ is a cubic threefold and 
${\cA_Z = \langle \cO_Z, \cO_Z(1) \rangle^\perp \subset \Db(Z)}$ is a fractional Calabi--Yau category of dimension $5/3$.
We again get a decomposition 
\begin{equation*}
\cA_X^{\bmu_3} = \langle \cA_Z, \cA_Z \otimes \chi \rangle.
\end{equation*}
So, we have a K3 category whose equivariant category decomposes into 
two copies of a fractional Calabi--Yau category of dimension $5/3$.
On the other hand,
applying the reconstruction result of Corollary~\ref{proposition-elagin}, 
we see the K3 category $\cA_X$ can be obtained as the equivariant
category of a category glued from two copies of the fractional Calabi--Yau category $\cA_Z$.
\end{itemize}

We note that the above construction can be iterated. 
For instance, consider a double cyclic cubic fourfold $X$, i.e.\ 
a cyclic cubic fourfold $X \to \bP^4$ such that the branch locus 
$Z \subset \bP^4$ is itself a cyclic cubic threefold. Concretely, 
in suitable coordinates $X$ is cut out in $\bP^5$ by an 
equation of the form
\begin{equation*}
F(x_0, \dots, x_5) = x_0^3 + x_1^3 - G(x_2,x_3,x_4,x_5).
\end{equation*}
The map $X \to \bP^4$ is given by projection onto the $x_1, \dots, x_5$ 
coordinates, and $Z \subset \bP^4$ is defined by $x_1^3 - G(x_2,x_3,x_4,x_5)$.
The group $\bmu_3 \times \bmu_3$ acts on $X$ by scaling the $x_0$ and $x_1$ 
coordinates, and the induced action on $\Db(X)$ preserves the decomposition 
\begin{equation*}
\Db(X) = \langle \cA_X, \cO_X, \cO_X(1), \cO_X(2) \rangle.
\end{equation*}
It follows from the definitions that there is an equivalence
\begin{equation}
\label{double-cyclic-cubic-sod-1}
\cA_X^{\bmu_3 \times \bmu_3} \simeq (\cA_X^{\bmu_3})^{\bmu_3},
\end{equation}
where on the right side the inner $\bmu_3$ acts 
by scaling on $x_0$ and the outer $\bmu_3$ acts by scaling on $x_1$.
By Theorem~\ref{main-theorem-precise} we have a decomposition
\begin{equation}
\label{double-cyclic-cubic-sod-2}
\cA_X^{\bmu_3} = \langle \Phi_0(\cA_Z), \Phi_1(\cA_Z) \rangle.
\end{equation}
It is straightforward to see the functors $\Phi_0, \Phi_1: \cA_Z \to \cA_X^{\bmu_3}$ 
are equivariant with respect to the $\bmu_3$-action on $\cA_Z$ 
(induced by the cyclic cover structure of $Z$) and the $\bmu_3$-action on
$\cA_X^{\bmu_3}$ described above. Hence, 
by a mild generalization of Elagin's result Theorem~\ref{theorem-group-action-sod}, 
we obtain
\begin{equation}
\label{double-cyclic-cubic-sod-3}
\cA_X^{\bmu_3 \times \bmu_3}  \simeq (\cA_X^{\bmu_3})^{\bmu_3} =
\langle \Phi_0(\cA_Z)^{\bmu_3}, \Phi_1(\cA_Z)^{\bmu_3} \rangle,
\end{equation}
where each component is equivalent to $\cA_Z^{\bmu_3}$. Combined with the description 
of $\cA_Z^{\bmu_3}$ above, we conclude $\cA_X^{\bmu_3 \times \bmu_3}$ is generated 
by an exceptional collection of length $32$.

Finally, we note that it is easy to see a double cyclic cubic fourfold $X$ contains 
a pair of skew planes. The results of~\cite{kuznetsov2010derived} then apply to show 
$\cA_X \simeq \Db(S)$ for a K3 surface $S$. Thus the above gives a description of the 
equivariant derived category of a K3 surface.

\providecommand{\bysame}{\leavevmode\hbox to3em{\hrulefill}\thinspace}
\providecommand{\MR}{\relax\ifhmode\unskip\space\fi MR }
\providecommand{\MRhref}[2]{%
  \href{http://www.ams.org/mathscinet-getitem?mr=#1}{#2}
}
\providecommand{\href}[2]{#2}

\end{document}